\documentclass[10pt]{article}
\usepackage[margin= 1.5in]{geometry}
\usepackage{verbatim}
\usepackage{amssymb}
\usepackage{amsthm}
\usepackage{amsmath}
\usepackage[usenames,dvipsnames] {color}
\usepackage{graphicx}
\usepackage[normalem]{ulem}
\usepackage{tikz}
\usetikzlibrary{matrix}
\numberwithin{equation}{section}
\usepackage[mathscr]{eucal}
\usepackage{pgfplots}

\title{A Noncommutative Borsuk-Ulam Theorem for Natsume-Olsen Spheres}
\author{Benjamin Passer\thanks{Partially supported by National Science Foundation Grants DMS 1300280 and DMS 1363250} \\ Washington University\\ St. Louis, MO 63130 \\ \\bpasser@math.wustl.edu}

\begin{document}

% Definitions

% Numbered Equations
\def\ba{\begin{aligned}}
\def\ea{\end{aligned}}
\def\be{\begin{equation}}
\def\ee{\end{equation}}
\def\beu{\begin{equation*}}
\def\eeu{\end{equation*}}

% Sets in Blackboard Bold
\def\C{\mathbb{C}}
\def\D{\mathbb{D}}
\def\R{\mathbb{R}}
\def\Rn{\mathbb{R}^n}
\def\S{\mathbb{S}}
\def\Sn{\mathbb{S}^n}
\def\Z{\mathbb{Z}}
\def\T{\mathbb{T}}
\def\N{\mathbb{N}}
\def\RP{\mathbb{RP}}
\def\Q{\mathbb{Q}}

% Transforms and Operators
\def\F{\mathfrak{F}}
\def\del{\partial}

% Other
\def\!={\neq}
\def\l{\langle}
\def\r{\rangle}

\theoremstyle{definition}

\newtheorem{defn}[equation]{Definition}
\newtheorem{lem}[equation]{Lemma}
\newtheorem{prop}[equation]{Proposition}
\newtheorem{thm}[equation]{Theorem}
\newtheorem{claim}[equation]{Claim}
\newtheorem{ques}[equation]{Question}
\newtheorem{fact}[equation]{Fact}
\newtheorem{axiom}[equation]{Technical Axiom}
\newtheorem{newaxiom}[equation]{New Technical Axiom}
\newtheorem{cor}[equation]{Corollary}
\newtheorem{exam}[equation]{Example}
\newtheorem{conj}[equation]{Conjecture}

\bibliographystyle{plain}
\maketitle

\begin{abstract}
Natsume-Olsen noncommutative spheres are $C^*$-algebras which generalize $C(\S^k)$ when $k$ is odd. These algebras admit natural actions by finite cyclic groups, and if one of these actions is fixed, any equivariant homomorphism between two Natsume-Olsen spheres of the same dimension induces a nontrivial map on odd $K$-theory. This result is an extended, noncommutative Borsuk-Ulam theorem in odd dimension, and just as in the topological case, this theorem has many (almost) equivalent formulations in terms of $\theta$-deformed spheres of arbitrary dimension. In addition, we present theorems on graded Banach algebras, motivated by algebraic Borsuk-Ulam results of A. Taghavi.
\end{abstract}

\section{Introduction}\label{sec:introBU}

The Borsuk-Ulam theorem in algebraic topology states that every continuous map $f: \mathbb{S}^n \to \Rn$ must admit some point $x$ on the sphere $\S^n$ such that $f(x) = f(-x)$. The standard proof (see \cite{ha02}) does not use this form of the theorem, but rather uses a reformulation in terms of maps between two spheres. First, decompose $f$ into even and odd components.
\be\label{eq:dec} \ba f(x) 	&= \frac{f(x) + f(-x)}{2} + \frac{f(x) - f(-x)}{2} \\
 					&:= e(x) + o(x)
\ea \ee
If $f(x)$ is never equal to $f(-x)$, then the map $g(x) = \frac{o(x)}{|o(x)|}$ is defined, odd, and maps $\S^n$ to $\S^{n - 1}$. The restriction of $g(x)$ to the equator $\S^{n - 1}$ is then odd and homotopically trivial. All of the arguments above are reversible, so the theorem has four equivalent forms.
\begin{thm}\label{thm:introBU}{[Borsuk-Ulam]} Each of the following conditions holds for $n \geq 2$.
\begin{enumerate}
  \item \label{classic} If $f: \S^n \to \Rn$ is continuous, then there is some $x \in \S^n$ with $f(x) = f(-x)$.
  \item\label{nohomo} If $o: \S^n \to \Rn$ is continuous and odd, then there is some $x \in \S^n$ with $o(x) = 0$.
  \item \label{dimension}There is no odd, continuous map $g: \S^n \to \S^{n - 1}$.
  \item\label{homo} If $h: \S^{n - 1} \to \S^{n - 1}$ is odd and continuous, then $h$ is homotopically nontrivial.
\end{enumerate}
\end{thm}

Recall that the degree of a continuous map $f: \S^k \to \S^k$ is defined in terms of the top homology $H_k(\S^k, \Z)$. Since $H_k(\S^k, \Z) \cong \Z$, the induced map $f_*$ on top homology is a homomorphism from $\Z$ to $\Z$, which corresponds to multiplication by a unique integer, called the degree. The degree may be equivalently defined in terms of cohomology, and a homotopically trivial map will have degree zero. The standard proof of the Borsuk-Ulam theorem, which proves version \ref{homo}, uses the following stronger condition. 

\be\label{odd degree}
\textrm{Any odd, continuous self-map of a sphere } \S^k \textrm{ must have odd degree.}
\ee

In the extremely interesting paper \cite{ta12}, A. Taghavi motivates the Borsuk-Ulam theorem in terms of graded algebras over finite abelian groups and presents a proof (and generalization) for the $\S^2$ case in this context. Perhaps the most distinguishing part of his proof is that it deals explicitly with formulation \ref{nohomo} of the theorem, and not formulation \ref{homo}, making particular use of the identification $\R^2 \cong \C$. The role of graded algebras is quite simple: the even/odd decomposition (\ref{eq:dec}) is an example of a grading on $C(\S^2) = C(\S^2, \C)$ by the group $\mathbb{Z}_2$. 

\begin{defn} If $A$ is a Banach algebra and $G$ is a finite group, then $A$ is \textit{G-graded} if it admits a decomposition $A = \bigoplus\limits_{g \in G} A_g$ into closed subspaces which satisfy $A_g \cdot A_h \subset A_{gh}$ for all $g, h \in G$. The elements of $A_g$ are called \textit{homogeneous}, and when $g \not= e$, \textit{nontrivial homogeneous}.
\end{defn}

For convenience, we assume every algebra has scalar field $\C$ and has unit denoted by $1$. When $G = \Z_n$, there is a clear group action by $\mathbb{Z}_n$ on $A$ associated to the grading, where $\omega$ is a primitive $n$th root of unity.

\be\label{grading map}
T: a = (a_0, \ldots , a_{n - 1}) \in A  \mapsto (a_0, \omega a_1, \ldots , \omega^{n - 1} a_{n - 1})
\ee

\noindent In other words, $A_i$ is prescribed as the eigenspace of $T$ for eigenvalue $\omega^i$, and as a result of the graded structure, such a map is not only linear, but also a continuous algebra isomorphism with $T^n = I$. The map $T$ then generalizes the action on $C(\S^1)$ sending $f(\cdot)$ to $f(\omega \cdot)$, and the action of $\Z_n$ on $A$ is described by $k \cdot a = T^k(a)$. Finally, the projections $\pi_j: A \to A_j$ take a form generalizing (\ref{eq:dec}).

\be\label{eq:decomposition}
a_j = \pi_j(a) = \frac{{a + \omega^{-j} \cdot Ta + \omega^{-2j} \cdot T^2a + \ldots + \omega^{-(n-1)j} \cdot T^{n - 1}a}}{n}
\ee

\noindent Of course, one may start with an action of $\mathbb{Z}_n$ and recover a grading by this formula.

 More generally, if $G$ is a compact, abelian, Hausdorff group which acts strongly continuously on a Banach algebra $A$ by $\alpha: G \to \textrm{Aut}(A)$, then for any $\tau$ in the Pontryagin dual $\widehat{G} = \{f: G \to \S^1: f \textrm{ is a continuous homomorphism}\}$, there is a corresponding homogeneous subspace $A_\tau$ defined as follows.
\be
A_\tau = \{a \in A: \textrm{for all } g \in G, \alpha_g(a) = \tau(g)a\}
\ee
If $\mu$ denotes the unique Haar measure on $G$ with $\mu(G) = 1$, then there is a homogeneous component projection defined by $\pi_\tau: A \to A_\tau$ defined by an integral formula.
\be\label{eq:Fourier}
\pi_\tau(a) = \int_G \tau(g^{-1}) \alpha_g(a) \,d\mu \in A_\tau
\ee

The integral above exists because its integrand is a continuous Banach-space valued function (and also bounded because $G$ is compact), and $\mu$ is a finite Borel measure. When the group in question is $\Z_n$, we have that $\widehat{\Z_n}$ is isomorphic to $\Z_n$, generated by a homomorphism which sends $1$ to a primitive $n$th root of unity, so the previous formula generalizes (\ref{eq:decomposition}). The map $a \mapsto (\pi_\tau(a))_{\tau \in \widehat{G}}$ is injective, but we should not expect a nice formula such as $a = \int_{\widehat{G}} \pi_\tau(a)$ (integrating over a suitable Haar measure) to cleanly generalize a graded decomposition $a = \pi_0(a) + \pi_1(a) + \ldots + \pi_{n - 1}(a)$ for a $\Z_n$ action, as such an overreaching statement would imply that every continuous function on the circle has a convergent Fourier series. In particular, $\widehat{\mathbb{S}^1} = \Z$ consists of the homomorphisms $z \mapsto z^n$, $n \in \mathbb{Z}$, and the natural action of $\S^1$ on $C(\S^1)$ by rotation produces the usual Fourier transform from (\ref{eq:Fourier}) in the sense that $\pi_n(f)$ is the function mapping $z \in \S^1$ to $\widehat{f}(n)z^n$. As such, the reconstruction of elements of $A$ from homogeneous components is a process enveloping all of the subtlety of Fourier series in the classical cases, and it is no surprise that dual groups provide a natural setting for a generalized Fourier transform. For more information on the role of group actions (in particular, free actions) on $C^*$-algebras, see \cite{ph87} and \cite{ph09}.

Back in the world of finite groups, Taghavi's proof in \cite{ta12} of the Borsuk-Ulam theorem for $\S^2$ uses $\Z_2$ graded structure from the antipodal map and his Main Theorem 1 to conclude that an odd function $f: \S^2 \to \C \setminus \{0\}$ would have no logarithm, contradicting the fact that the exponential map qualifies $\C$ as the universal cover of $\C \setminus \{0\}$. In section \ref{sec:graded}, we prove a few new results in the same spirit as Taghavi\rq s, focusing on roots instead of logarithms, and relaxing some conditions on the Banach algebra $A$ and its idempotents. 

Next, in section \ref{sec:setup}, we introduce the obvious $\Z_2$ action on Natsume-Olsen odd spheres $C(\S^{2n - 1}_\rho)$, which are $C^*$-algebras defined by T. Natsume and  C. L. Olsen in \cite{na97}, generalizing the work of K. Matsumoto (\cite{ma91}) in dimension three. These spheres are also called $\theta$-deformed (odd) spheres, as they may be reached through M. Rieffel's quantization procedure in \cite{ri93} from an $\R^n$ action which factors through the torus $\T^n$. The main goal is to prove a noncommutative Borsuk-Ulam theorem for these spheres, as M. Yamashita did for the $q$-deformed spheres in \cite{ya13}. We consider different versions of the Borsuk-Ulak theorem as potential candidates for generalization, but some simple counterexamples show that viewing odd functions $\S^k \to \R^k$ in terms of odd elements of the algebra $C(\S^k)$ without noncommutation relations is a fruitless endeavor. However, statement \ref{homo} and (\ref{odd degree}) do generalize nicely to the noncommutative setting using $K$-theory (which aligns well with the $q$-deformed case in \cite{ya13}). This is proved in section \ref{sec:BUR} as Corollary \ref{cor:BU}, repeated here.

\begin{cor}\label{thm:introZ2BU}
Suppose $\Phi: C(\S^{2n - 1}_\rho) \to C(\S^{2n - 1}_\omega)$ is a unital $*$-homomorphism between two Natsume-Olsen spheres of the same dimension. If $\Phi$ is equivariant for the antipodal action, then $\Phi$ induces a nontrivial map on $K_1 \cong \Z$. More precisely, $\Phi_*: \Z \to \Z$ is multiplication by an odd integer.
\end{cor}

The Natsume-Olsen spheres are formed from the commutative sphere $C(S^{2n - 1})$ by $\theta$-deformation, so their $K$-groups have isomorphisms described in \cite{ri93b} to $K_j(C(\S^{2n - 1}))$. Natsume and Olsen chose to specify $K_1 \cong \Z$ more concretely in terms of a noncommutative Toeplitz algebra, and we adopt this identification: an invertible matrix $M$ over $C(\S^{2n - 1}_\rho)$ is indentified with the negative index of its Toeplitz operator. Next, the $\theta$-deformed even spheres $C(\S^{2m}_\rho)$ may be be found in \cite{co02}, and they are described via generators and relations in \cite{pe13} (among other places) with some results on projective modules. These spheres also admit a natural antipodal action, giving us a corollary (Corollary \ref{cor:dimensionNCBU}) of the above result.

\begin{cor}\label{cor:introdimBU}
There is no unital $*$-homomorphism $\Psi: C(\S^{k-1}_\rho) \to C(\S^{k}_\omega)$ which is equivariant for the antipodal action.
\end{cor}

\noindent More generally, there is no equivariant map from $C(\S^{n}_\rho)$ to $C(\S^{m}_\omega)$ when $n < m$. This is analagous to a result of M. Yamashita in $\cite{ya13}$ on $q$-deformed spheres; in the following theorem, the $\Z_2$ action is a generalized antipodal map.

\begin{thm}[Yamashita] For any $0 < q \leq 1$ and any integers $n < m$, there is no $\Z_2$-equivariant unital $*$-homomorphism from $C(\S^n_q)$ to $C(\S^m_q)$.
\end{thm}

The processes of $q$-deformation and $\theta$-deformation produce distinct families of spheres, and the techniques of proof for our Borsuk-Ulam theorems are different, relying on results about fixed point subalgebras at the end of this section. Now, any even $\theta$-deformed sphere $C(\S^{2n}_\rho)$ may also be realized as the unreduced suspension $\Sigma C(\S^{2n - 1}_\rho) = \{f: [-1, 1] \to C(\S^{2n - 1}_\rho): f \textrm{ is continuous and } f(-1), f(1) \in \C\}$, which places Corollary \ref{cor:introdimBU} in the context of a conjecture from \cite{da15}. 

\begin{conj}[Dabrowski]
If $A$ is a unital $C^*$-algebra with a free $\mathbb{Z}_2$ action, then there is no equivariant [unital] $*$-homomorphism from $A$ to $\Sigma A$. [$\Sigma A$ admits a $\Z_2$ action from composing the pointwise action of $\Z_2$ on $A$ with $t \mapsto -t$ on the domain $[-1, 1]$.]
\end{conj}

This conjecture is tangential to other work on generalizing sphere theorems; see \cite{ba15} for conjectures and examples on noncommutative joins by P. Baum, L. Dabrowski, and P. Hajac. Now, a different extension of the previous corollaries can be reached within $C(\S^{2n - 1}_\rho)$; these odd spheres admit rotation maps which generalize $(z_1, \dots z_n) \mapsto (\alpha_1z_1, \ldots, \alpha_n z_n)$ on $S^{2n - 1} \subset \mathbb{C}^n$ for any $\alpha_i \in \S^1$, not just for $\alpha_i = -1$. If $R$ denotes the generalization of this rotation to the Natsume-Olsen spheres when each $\alpha_i$ is a primitive root of unity of the same order $k \geq 2$, then we have the following result from Corollary \ref{cor:ZnBU}.

\begin{cor}
Suppose $\Phi: C(\S^{2n - 1}_\rho) \to C(\S^{2n - 1}_\omega)$ is a unital $*$-homomorphism which is equivariant for $R$ (of order $k \geq 2$). Then $\Phi_*$ is nontrivial on $K_1 \cong \Z$, given by multiplication by an integer in $k \mathbb{Z} + 1$.
\end{cor}
While the above results all concern homomorphisms on spheres, they are proved using a theorem on fixed point subalgebras. The most general form used is as follows, from Theorem \ref{thm:ZnNCcase}.
\begin{thm}
Let $R$ (as above) have order $k \geq 2$, and suppose $U \in \mathcal{U}_d(\mathbb{C})$ is a unitary matrix with order dividing $k$. If $M$ is an invertible matrix over $C(\mathbb{S}^{2n - 1}_\rho)$ with $U R(M) U^* = M$, then the equivalence class of $M$ in $K_1(C(\S^{2n - 1}_\rho)) \cong \Z$ is an element of $k \Z$.
\end{thm}
The relevance of this theorem to Borsuk-Ulam type results comes from the fact that $K_1(C(\S^{2n - 1}_\rho)) \cong \Z$ may be written with a generator that is nontrivial homogeneous for numerous rotation-and-conjugation actions.

\section{Graded Banach and $C^*$-algebras}\label{sec:graded}

Below is Main Theorem 1 of Taghavi in \cite{ta12}, in which $k$ is a positive integer and $G$ is a finite abelian group. It is proved by reducing to the $\Z_n$ case by quotient groups.

\begin{thm}\label{thm:tag}(\cite{ta12}, Main Theorem 1)
Let $A$ be a $G$-graded Banach algebra [$G$ is finite and abelian] with no nontrivial idempotents. Let $a \in A$ be a nontrivial homogeneous element. Then $0$ belongs to the convex hull of the spectrum $\sigma(a^k)$. Further, if $A$ is commutative and $a$ is invertible, then $a^k$ and $1$ do not lie in the same connected component of the space of invertible elements $G(A)$. 
\end{thm}

Note in particular that there are no restrictions on $k \in \Z^+$; for example, $a^k$ might be a trivial homogeneous element. If $A$ is equal to $C(X)$ for a compact Hausdorff space $X$, then $X$ is connected if and only if $A$ has no nontrivial idempotents. The spectrum result in Taghavi's theorem illustrates the following problem: if $a$ is an invertible element that is nontrivial homogeneous, then in some $\Z_n = G/N$ grading with associated isomorphism $T$ and primitive $n$th root of unity $\omega$, $T(a) = \omega a$. Since $\sigma(a) = \sigma(Ta) = \sigma(\omega a) = \omega \sigma(a)$, if $\sigma(a)$ is missing values in any particular ray $e^{i\theta}[0, \infty)$, rotational symmetry will disconnect $\sigma(a)$ into $n$ pieces. The holomorphic functional calculus then provides a nontrivial idempotent in $A$, which contradicts the assumptions. This is a proof of a more general spectral condition than Taghavi claims: the connected set $\sigma(a)$ will either include $0$ or completely surround $0$ in $\C$, so we should not expect a logarithm of $a$ (or of $a^k$) using functional calculus. Taghavi's full result is a statement about (nonexistence of) logarithms that is not limited to functional calculus, and we have listed below the most general result that may be clearly distilled from the original proof; this also resolves our petty quibbles about the spectrum.

\begin{thm}\label{restatement}(Restatement of \cite{ta12}, Main Theorem 1)
Let $A$ be a $G$-graded Banach algebra with no nontrivial idempotents, where $G$ is a finite abelian group, and suppose $a \in A$ is a nontrivial homogeneous element. If $k \in \Z^+$, then there is no $b \in A$ with the following properties. 
\begin{enumerate}
  \item $g, h \in G \implies b_g b_h = b_h b_g$ 
  \item $ab = ba$
  \item $\exp(b) = a^k$
\end{enumerate}
\end{thm}

If we return to the motivating example of functional calculus, the same topological obstruction on the spectrum occurs when trying to form $n$th roots of invertible elements instead of logarithms, so one can ask if similar results hold for roots. Some simple counterexamples show that there must be a relationship between the size of the group $\Z_n$ and the order of the root, so these results are more algebraic in motivation than analytic. 

\begin{prop}\label{prop:tagish}
Suppose $A$ is a $\Z_n$-graded Banach algebra with no nontrivial idempotents. If $a$ is a nontrivial homogeneous element that is also invertible, then $a$ cannot have an $n$th root $b$ such that all the homogeneous components $b_k$ commute.
\end{prop}
\begin{proof}
Suppose $b$ is such an $n$th root of $a$, so that $b$ is also invertible and commutes with $a$. Consequently, if $T$ is the isomorphism associated to the graded algebra such that $T(a) = \omega^j a$, then the fact that the homogeneous components $b_k$ of $b$ all commute implies that $b^{-1}$ and $Tb$ commute. This shows that $(b^{-1} Tb)^n = b^{-n} T(b^n)$, which is equal to $a^{-1} Ta = \omega^j$. Now, $b^{-1}Tb$ is an $n$th root of a constant, so by the spectral mapping theorem, its spectrum is finite. Also, the spectrum must be connected because $A$ has no nontrivial idempotents, so $\sigma(b^{-1}Tb) = \{c\}$ and $b^{-1}Tb = c + \varepsilon$, where $\varepsilon$ is quasinilpotent ($\sigma(\varepsilon) = \{0\}$) and $c^n = \omega^j$. 

All elements that follow are in the closed subalgebra generated by elements of the form $T^k b$ or $T^k (b^{-1})$, which is commutative. The equation $b^{-1}Tb = c + \varepsilon$ implies that $Tb = b(c + \varepsilon)$, and an inductive argument shows that $T^k b = b \cdot \prod\limits_{j = 1}^k (c + T^{j - 1}\varepsilon)$. When $k = n$ this says $b = T^n b = b \cdot \prod\limits_{j = 1}^n (c + T^{j - 1}\varepsilon)$. Since $\varepsilon$ is quasinilpotent, each $T^{j - 1} \varepsilon$ is quasinilpotent, and the commuting product $\prod\limits_{j = 1}^n (c + T^{j - 1}\varepsilon)$ is equal to $c^n + \delta$ = $\omega^j + \delta$ where $\delta$ is quasinilpotent. The element $\delta$ commutes with $b$, so $b = b \cdot \prod\limits_{j = 1}^n (c + T^{j - 1}\varepsilon) = b(\omega^j + \delta) = b\omega^j + \gamma$ where $\gamma$ is quasinilpotent. Finally, $a$ was a nontrivial homogeneous element, so $1 - \omega^j \not= 0$, and $(1 - \omega^j)b = \gamma$ is both invertible (as $b$ is invertible) and quasinilpotent. This is a contradiction.
\end{proof}

The proof technique for the previous proposition is directly inspired by Taghavi's methods. Invertibility of the element $a$ and the relationship between the order of the group $\Z_n$ and the order of the root cannot be removed. These requirements can be seen in the commutative algebra $C(\S^1)$ with the standard $\Z_2$ antipodal action.

\begin{exam}
If $\S^1$ is realized as the unit sphere of $\R^2$, then the coordinate functions $x_1$ and $x_2$ in $C(\S^1)$ are odd. Since $\sigma(x_i) = [-1,1]$ and $x_i$ is a normal element of a $C^*$-algebra, we may apply the continuous functional calculus for the following square root function.

\beu
g(t) =
\begin{cases}
\sqrt{t}, & t \in [0, 1] \\
i\sqrt{-t}, & t \in [-1, 0]
\end{cases}
\eeu
Now, $g(x_i)$ is a square root of the (non-invertible) odd element $x_i$.
\end{exam}

\begin{exam}
The invertible odd function $f(z) = z^3$ in $C(\S^1)$ certainly has a third root.
\end{exam}

The previous proposition still assumes that $A$ has no nontrivial idempotents, which can be problematic when $A$ is a noncommutative $C^*$-algebra. For $\Z_2$-graded Banach algebras this can be resolved by modifying the original proof to construct an idempotent.

\begin{thm}\label{thm:thenumber23}
Suppose $A$ is a $\Z_2$-graded Banach algebra with the property that no idempotent $P$ satisfies $T(P) = 1 - P$. Then if $f \in A$ is odd and invertible, there is no $g \in A$ such that $g^2 = f$ and $g$ commutes with $Tg$.
\end{thm}
\begin{proof}
Suppose $g^2 = f$ where $g$ and $Tg$ commute. Then $g$ is invertible and 

\beu
(T(g) g^{-1})^2 = T(g^2)(g^2)^{-1} = T(f) f^{-1} = -1
\eeu
holds. Denote the element $T(g) g^{-1}$ by $a$ and note that $a^2 = -1$, so $a^{-1} = -a$. However, we also have that $T(a) = -a$ by a simple calculation.
\beu
T(a) = T(T(g) g^{-1}) = g T(g)^{-1} = (T(g)g^{-1})^{-1} = a^{-1}= -a
\eeu
This means $a$ is odd, so $a$ is an odd square root of $-1$. It follows that $P = \frac{1}{2} + \frac{i}{2} a$ is an idempotent with $T(P) = 1 - P$.
\end{proof}

The condition $T(P) \not= 1 - P$ is not only sufficient in the above theorem, but also necessary. If $T(P) = 1 - P$, then $\pi_0(P) = \frac{P + T(P)}{2} = 1/2$, so if we examine the odd component $\pi_1(P) = b$, the idempotent equation $(1/2 + b)^2 = 1/2 + b$ implies that $b^2 = 1/4$. Consequently, $\sigma(b)$ is finite (and excludes 0) by the spectral mapping theorem. We may then form a square root $c$ of the invertible odd element $b$ by the holomorphic functional calculus. Since $b$ is odd and $c$ is in the closed, unital subalgebra generated by $b$ and elements of the form $(b - \lambda)^{-1}$, it follows that $c T(c) = T(c) c$.

For a $\Z_2$ action on a $C^*$-algebra, if we assume $T(P) \not= 1 - P$ on the smaller class of projections (instead of all idempotents), then we obtain a similar result with a slightly weaker conclusion.

\begin{thm}\label{thm:letsdothatagain}
Suppose $A$ is a $C^*$-algebra with a (*-compatible) $\Z_2$ action such that no projection $P$ satisfies $T(P) = 1 - P$. Then if $f \in A$ is an odd unitary element, there is no unitary $g \in A$ such that $g^2 = f$ and $g$ commutes with $Tg$.
\end{thm}
\begin{proof}
The proof is the same as the proof of the previous theorem, with the addition that since $g$ is unitary, $a = T(g)g^{-1} = T(g)g^*$ satisfies $a^* = a^{-1} = -a$, and the resulting $P$ is self-adjoint.
\end{proof}

\noindent \textit{Remark.} As in the previous theorem, the condition $T(P) \not= 1 - P$ is also necessary here. The only change to the argument is that the odd component $b$ of a projection satisfying $T(P) = 1 - P$ is also self-adjoint, which with the equation $b^2 = 1/4$ implies that $2b$ is unitary. Again, this element has finite spectrum, and the square root formed by the continuous functional calculus is guaranteed to be unitary.

Since the homogeneous subspaces $A_0$ and $A_1$ of a $C^*$-algebra with a $\Z_2$ action are norm-closed and closed under the adjoint operation, any even or odd element $a$ has $a a^*$ and $a^* a$ even, and the positive square root of either $a a^*$ or $a^* a$ from the continuous functional calculus is even as well (as a limit of polynomials in an even element). Similarly, the inverse of an even or odd element remains even or odd, as seen by examining the effect of the isomorphism $T$. These observations show that if we start with a homogeneous invertible and scale it to form a unitary, the result is still homogeneous, giving some equivalent formulations of the projection condition.

\begin{prop}\label{eqPforms} The following conditions are equivalent for a $C^*$-algebra $A$ with a $\Z_2$ action defined by isomorphism $T$.
\begin{enumerate}
\item\label{original}There is a projection $P \in A$ with $T(P) = 1 - P$.
\item\label{oddunitary}There is some $a \in A$ which is odd, self-adjoint, and satisfies $a^2 = 1$.
\item\label{oddinvertible}There is some $b \in A$ which is odd, self-adjoint, and invertible.
\end{enumerate}
\end{prop}
\begin{proof}
Condition $\ref{oddunitary}$ certainly implies condition $\ref{oddinvertible}$, and the reverse implication holds by scaling $b$ to a unitary $a = b (b^2)^{-1/2} = b |b|^{-1}$, which remains odd and self-adjoint. If $P$ is a projection with $T(P) = 1 - P$, then its even component is $\frac{1}{2}(P + T(P)) = 1/2$, so $P$ is of the form $1/2 + c$, where $c$ is self-adjoint and odd. The idempotent equation $(1/2 + c)^2 = 1/2 + c$ implies that $c^2 = 1/4$, so $a = 2c$ satisfies $a^2 = 1$, and condition \ref{original} implies condition \ref{oddunitary}. Similarly, if $a$ is as in condition \ref{oddunitary}, then $P = 1/2 + a/2$ is a projection with $T(P) = 1 - P$.
\end{proof}

The condition $T(P) \not= 1 - P$ allows for some projections to exist in the algebra $A$. As an example, the quantum $n$-torus $A_\theta$, for $\theta$ an antisymmetric $n \times n$ matrix over $\R$, is generated by unitaries $U_1, \ldots, U_n$ satisfying the following noncommutativity condition.

\be\label{NCtorus} U_k U_j = e^{2 \pi i \theta_{jk}} U_j U_k \ee

For most values of $\theta$, $A_\theta$ has nontrivial projections. The algebra is also a well-established example of a deformation quantization (see \cite{ri93}, Chapter 10) of $C(\T^n)$. In the language of M. Rieffel in \cite{ri93}, $A_\theta = C(\T^n)_J$, where $J$ is the antisymmetric matrix $\theta/2$ and $C(\T^n)$ is equipped with an $\R^n$ action defined by translation in angular coordinates. Each $A_\theta$ contains the common dense subalgebra $C^\infty(\T^n)$ acting under different products $\cdot_\theta$ and norms $|| \cdot ||_\theta$, but with the same linear structure, adjoint, and multiplicative identity. The unitary functions $u_p \in C^\infty(\T^n)$ for $p \in \Z^n$, defined by $u_p(w_1, \ldots,w_n) = w_1^{p_1}\cdots w_n^{p_n}$, are in spectral subspaces for the $\R^n$ action, so they satisfy a relation tying $\cdot_\theta$ to the usual commutative product.

\be\label{specprod}
u_p \cdot_\theta u_q = e^{\pi i [(\theta p) \cdot q]} u_{p + q}
\ee

This is more general than the relation $u_p \cdot_\theta u_q = e^{2\pi i [(\theta p) \cdot q]} u_q \cdot_\theta u_p$. Moreover, the generators $U_1, \ldots, U_n$ of any $A_\theta$ are of this form: $U_1 = u_{(1, 0, \ldots, 0)}$, $U_2 = u_{(0, 1, 0, \ldots, 0)}$, and so on. In general the relationship in (\ref{specprod}) between the product $\cdot_\theta$ and the usual commutative product $u_p u_q = u_{p + q}$ shows that the antipodal map on $C^\infty(\T^n)$ defines a $\Z_2$ structure that is simultaneously compatible with each product $\cdot_\theta$. This is a result of the fact that the antipodal map on $C(\T^n)$ commutes with the $\R^n$ action of translation in angular coordinates, which defines the quantization. Any $*$-polynomial under $\cdot_\theta$ in the generators $U_i$ can then be written as a linear combination $\sum a_p u_p$ by pushing to the commutative product, and the $\Z_2$ action takes the form $T(\sum a_p u_p) = \sum (-1)^{p_1 + p_2 + \ldots + p_n} a_p u_p$. Now, the $\Z_2$-graded algebras $A_{\theta + \hslash \phi}$ also form a strict deformation quantization (\cite{ri93}, Definition 9.2, Theorem 9.3), leading to the following continuity assertions for fixed $f, g \in C^\infty(\T^n)$. 
\be\label{torusnormfixed}
\lim_{\hslash \to 0} ||f||_{\theta + \hslash \phi} = ||f||_\theta
\ee

\be\label{torusprodnorm}
\lim_{\hslash \to 0} ||f \cdot_{\theta + \hslash \phi} g - f \cdot_\theta g||_{\theta + \hslash \phi} = 0
\ee

These limits do not use the full strength of strict quantization, but even so, they will interact with the common $\Z_2$ structure on $A_\theta$ to help show that each $A_\theta$ has $T(P) \not= 1 - P$ for all projections. Note that since $C^\infty(\T^n)$ is $T$-invariant and $T$ is $*$-compatible, we can approximate homogeneous elements in $A_\theta$ with homogeneous elements of $C^\infty(\T^n)$, where we may demand the approximations remain self-adjoint if the $A_\theta$ element is self-adjoint. By the previous comments, these smooth elements remain (self-adjoint and) homogeneous when viewed in different noncommutative tori.

\begin{prop}\label{prop:2torus} There is no projection $P$ with $T(P) = 1 - P$ in any quantum $n$-torus $A_\theta$.
\end{prop}
\begin{proof}

Suppose for some $\theta$ there is a projection $P \in A_\theta$ with $T(P) = 1 - P$, which by Proposition \ref{eqPforms} means there is a self-adjoint odd element $a \in A_\theta$ with $a \cdot_\theta a = 1$. Approximate $a$ with a self-adjoint, odd element $b \in C^\infty(\T^n)$ which has $||b \cdot_\theta b - 1||_\theta < 1$. When the parameter of the algebra $A_\theta$ changes, $b$ remains odd and self-adjoint. Perturb the entries of the antisymmetric matrix $\theta$ using (\ref{torusprodnorm}) and (\ref{torusnormfixed}) multiple times to replace $\theta$ with an antisymmetric $\psi$, where each entry of $\psi$ is rational and has odd denominator. In particular, $b \in C^\infty(\T^n)$ is still odd, self-adjoint, and invertible as an element of $A_\psi$.

We may form a homomorphism from $A_\psi$ to a matrix algebra over $C(\T^n)$ in a way similar to \cite{ma12}. First, since $\psi$ is rational and antisymmetric, an inductive argument shows there are unitaries $V_1, \ldots, V_n$ in some $U_q(\C)$ such that $V_k V_j = e^{2 \pi i \psi_{jk}} V_j V_k$ for all $j$ and $k$. Moreover, since the denominator of each $\psi_{jk}$ is odd, we may form these matrices so that the dimension $q$ of the matrix algebra is an odd integer. The universal property of $A_\psi$ shows that there is a $*$-homomorphism

\beu
E: A_\psi \to M_q(C(\T^n))
\eeu
\beu
U_j \mapsto w_jV_j
\eeu
where $w_j \in C(\T^n)$ is the $j$th coordinate function. Since the generators $U_j$ of $A_\phi$ are odd, and their images $w_j V_j$ have odd functions in every entry, the map $E$ is equivariant for the (entrywise) antipodal maps. The image of $b \in A_\psi$ is then a self-adjoint, invertible matrix of odd dimension $q$, with each entry an odd function on $\T^n$. The determinant of this matrix is a nowhere vanishing, real-valued, odd function on $\T^n$, which gives a contradiction since $\T^n$ is connected.
\end{proof} \noindent\textit{Remark.} This argument applies equally well to $M_{2k + 1}(A_\theta)$. Counterexamples can be easily constructed for $M_{2k}(A_\theta)$, such as $P = \frac{1}{2}I_2 + \frac{1}{2}\left[ \begin{array}{cc}
0 & U_1\\
U_1^*  & 0\\
 \end{array} \right]$.

\vspace{.05 in}

As alluded to in the above proof, when $A$ is a graded Banach algebra, $M_n(A)$ is graded as well; the homogeneous subspaces consist of matrices with entries in the homogeneous subspaces of $A$. However, $M_n(A)$ will always have nontrivial idempotents for $n \geq 2$, so Taghavi's Main Theorem 1 in \cite{ta12} and the similar result Proposition \ref{prop:tagish} in this section do not apply. However, the new condition $T(P) \not= 1 - P$ allows for some idempotents, and the matrix dimension will play a key role. For example, if $A$ is a $C^*$-algebra and there exists an $n \times n$ unitary matrix $F$ over $A$ which has odd entries, then $P = \left[ \begin{array}{cc}
1/2 & F/2\\
F^*/2  & 1/2\\
 \end{array} \right]$ is a projection in the $2n \times 2n$ matrix algebra with $T(P) = I - P$. 

The condition $T(P) \not= 1 - P$ has a simple restatement when the algebra is $C(X)$ where, say, $X$ is compact and has finitely many components, and $T$ arises from a continuous $\Z_2$ action on $X$ (written as $x \mapsto -x$). In this case, the $\Z_2$ action pairs up the connected components of $X$, where sometimes a component pairs with itself. If no component pairs with itself, then group the finitely many components into two disjoint pieces $X_1$ and $X_2$ separating these pairs and define a function which is zero on $X_1$ and one on $X_2$. This projection satisfies $T(P) = 1 - P$. Insisting that $T(P)$ is never $1 - P$ then means at least one component pairs with itself. In this case, the quotient algebra of functions on this component reduces the problem to the idempotentless case, so the actual benefit of the new condition is for noncommutative algebras (for example, $A_\theta$ above, which fundamentally has nontrivial projections). In $M_n(C(X))$, a projection $P$ assigns to each $x \in X$ a projection $P_x \in M_n(\C)$, which as a linear map is the orthogonal projection onto a subspace of $\C^n$, forming a continuous vector bundle. If $M_n(C(X))$ inherits the $\Z_2$ action from $X$ and $T(P) \not= I - P$, then there is some $x$ with $P_{-x} \not= I - P_x$. This means the vector bundle assigns some point to a subspace other than the orthogonal complement of the subspace assigned to its opposite point.

A stronger version of the condition demands that if $P$ is a projection and $T(P)P = 0$, then $P = 0$. In $C(X)$ as above, this means that every component of $X$ must pair with itself under the $\Z_2$ action. For $M_n(C(X))$, if $P$ is a nonzero projection (vector bundle), some $x$ must have $P_{-x} P_{x} \not= 0$, meaning the subspaces assigned to pairs of opposite points must not always be orthogonal to each other. This requirement allows for a stronger version of Theorem \ref{thm:thenumber23}, in which the odd invertible element that allegedly has no square root is replaced by a projection plus an odd element. This type of element occurs frequently in $K$-theory, as a unitary matrix $F$ over a $C^*$-algebra may have odd entries, but $F \oplus I$ does not.

\beu \ba
\textrm{ Projection + Odd } =  \left[ \begin{array}{cc}
0 & 0\\
0  & I\\
 \end{array} \right] + \left[ \begin{array}{cc}
F & 0\\
0  & 0\\
 \end{array} \right] = \left[ \begin{array}{cc}
F & 0\\
0  & I\\
 \end{array} \right]
\ea \eeu

\begin{thm}
Suppose $A$ is a $C^*$-algebra with $\Z_2$ action generated by $T$ such that every nonzero projection $P$ has $T(P)P \not= 0$. If $f$ is a nonzero odd element and $\alpha$ is a projection such that $\alpha + f$ is unitary and $\alpha f = f \alpha = 0$, then there is no unitary $g$ such that $gT(g) = T(g)g$ and $g^2 = \alpha + f$.
\end{thm}
\begin{proof}
The conditions imply that $\alpha f^* = f^* \alpha = 0$ and $\alpha + f f^* = \alpha + f^* f = 1$, the last of which shows that $\alpha$ is even. Suppose $g$ is unitary with $g^2 = \alpha + f$ and $gT(g) = T(g)g$, so $T(g)$ and $g^{-1}$ also commute.
\beu
(T(g) g^{-1})^2 = T(g^2) g^{-2} = T(\alpha + f) (\alpha + f)^*  = (\alpha - f)(\alpha + f^*) = \alpha - ff^* = 2 \alpha - 1
\eeu

Since $\alpha$ is a projection, the spectrum of $2 \alpha - 1$ is contained in $\{-1, 1\}$. The spectral mapping theorem then implies that $\sigma(T(g)g^{-1}) \subset \{i, -i, -1, 1\}$. Moreover, $T(g)g^{-1}$ is unitary, so the continuous functional calculus is applicable, giving the following decomposition.

\beu 
T(g) g^{-1} = iP - iQ -R +S
\eeu

\noindent Here $P, Q, R,$ and $S$ are mutually orthogonal projections with $P + Q + R + S = 1$. Also, $T(g) g^{-1}$ is a unitary element $b$ such that $T(b) = b^{-1} = b^*$, so the commutative $C^*$-algebra it generates is also $T$-invariant, meaning all projections in the following computations commute with each other.

Next, we rephrase the above equation as

\be \label{Tgiswhat}
T(g) = (iP - iQ - R + S)g
\ee

\noindent and apply $T$ to both sides.
\beu \ba
g 	&= (iT(P) - iT(Q) - T(R) + T(S))T(g) \\
	&= (iT(P) - iT(Q) - T(R) + T(S))(iP - iQ - R +S)g
\ea \eeu

\noindent Multiplying out this expression and canceling the invertible element $g$ give that

\be\label{beta}
1 = (-1)\beta_{-1} + i\beta_i + (-i)\beta_{-i} + \beta_1
\ee
where the $\beta$ terms are sums of products of projections.
\beu \ba
&\beta_{-1} 	&= \hspace{.3 cm} T(P)P + T(Q)Q + T(R)S + T(S)R \\
&\beta_i 	&= \hspace{.3 cm} T(P)S + T(Q)R + T(R)Q + T(S)P \\
&\beta_{-i}	&= \hspace{.3 cm} T(P)R + T(Q)S + T(R)P + T(S)Q \\ 
&\beta_{1}	&= \hspace{.3 cm} T(P)Q + T(Q)P + T(R)R + T(S)S
\ea \eeu

Each of the sixteen commuting products of two projections above is a projection. Moreover, any two of these sixteen projections annihilate each other: for example, $T(R)P \cdot T(Q)P = T(Q) P \cdot T(R)P = 0$ because $RQ = QR = 0$. This means that each $\beta$ term is a projection, and the four projections are mutually orthogonal. Equation (\ref{beta}) then implies that $\beta_{-1} = 0$, so $T(P)P + T(Q)Q  + T(R)S + T(S)R = 0$. As each of the terms adding to zero is a projection and therefore positive, $T(P)P = 0 = T(Q)Q$, and by the assumption on the algebra, $P = Q = 0$. Finally, this gives a simpler form of (\ref{Tgiswhat}).

\beu \ba
T(g) = (iP - iQ - R + S)g = (-R + S)g
\ea \eeu

The projections $R$ and $S$ annihilate each other and commute with $g$, as they are in the $C^*$-algebra generated by $T(g) g^{-1}$, and further $R + S = P + Q + R + S = 1$. Last, we square both sides of $T(g) = (-R + S)g$ to reach that $f = 0$.

\beu \ba
T(g^2) &= (-R + S)^2 g^2\\
T(\alpha + f) &= (R + S)(\alpha + f) \\
\alpha - f &= \alpha + f \\
f &= 0
\ea \eeu

\end{proof}

\section{Natsume-Olsen Spheres}\label{sec:setup}

The Gelfand-Naimark Theorem (\cite{do98}, Theorem 4.29) states that any commutative $C^*$-algebra $A$ with unit can be written uniquely as $C(X)$ for some compact Hausdorff space $X$. Moreover, this relationship forms a contravariant functor: continuous maps $X \to Y$ correspond to unital $*$-homomorphisms $C(Y) \to C(X)$. So, when one discusses a noncommutative topological space, such as the noncommutative torus or a noncommutative sphere, one means some noncommutative $C^*$-algebra which shares many relevant properties of $C(X)$ for that choice of $X$. A prototypical example of this is the noncommutative $n$-torus $A_\theta$, where $\theta$ is an $n \times n$ antisymmetric matrix of real numbers, as used in Proposition \ref{prop:2torus}. When $\theta$ is an integer matrix, $U_1, \ldots, U_n \in A_\theta$ are just commuting unitaries, and $A_\theta$ is equal to $C(\T^n)$. As in \cite{na97}, the noncommutative torus may be defined slightly differently as $A_\rho$ with a coordinate change $\rho_{jk} = e^{2 \pi i \theta_{jk}}$, so that $U_k U_j = \rho_{jk} U_j U_k$. From now on, we will use this convention of $\rho$-coordinates when discussing tori and spheres. To avoid restating the properties of $\rho$, we give the following definition. 

\begin{defn}
An $n \times n$ matrix $\rho$ is called a \textit{parameter matrix} if $\rho$ is self-adjoint, all entries of $\rho$ are unimodular, and $\rho_{ii} = 1$ for all $1 \leq i \leq n$. Such matrices may be written (nonuniquely) as $\rho_{jk} = e^{2 \pi i \theta_{jk}}$ where $\theta$ is real and antisymmetric.
\end{defn}

Any quantum torus $A_\rho$ admits a $\Z_2$ action defined from the following unital $*$-homomorphism $T$, which generalizes the antipodal map.

\beu
T: A_\rho \to A_\rho
\eeu
\beu
U_j \mapsto -U_j
\eeu

\noindent This map satsifies $T^2 = 1$, and $T$ exists because $-U_1, \ldots, -U_n$ satisfy the relations defining $A_\rho$; they are unitary and $(-U_k)(-U_j) = \rho_{jk} (-U_j)(-U_k)$. Note that this $\Z_2$ action is the same as the action obtained through quantization, as in the discussion before Proposition \ref{prop:2torus}. This is not surprising since the antipodal map on $C(\T^n)$ is seen in the $\R^n$ action defining the quantization, namely as the action of a finite subgroup of $\R^n / \Z^n$.  Much of the same structure is present in the noncommutative spheres of Natsume and Olsen in \cite{na97}, which are also defined by generators and relations, or equivalently by deformation quantization.

\begin{defn} If $\rho$ is an $n \times n$ parameter matrix, the Natsume-Olsen odd sphere $C(\S^{2n - 1}_\rho)$ is the universal, unital $C^*$-algebra generated by elements $z_1, \ldots, z_n$ subject to the following relations.
\be\label{normal} z_j z_j^* = z_j^* z_j \ee
\be\label{NCgen} z_k z_j = \rho_{jk}z_j z_k \ee
\be\label{SS} z_1 z_1^* + z_2 z_2^* + \ldots + z_n z_n^* = 1 \ee
\end{defn}
A nontrivial result in \cite{na97} gives a noncommutativity relation for $z_j$ and $z_k^*$ as a consequence of the above definition.
\be \label{NCadj}
z_k^* z_j = \overline{\rho_{jk}} z_j z_k^* = \rho_{kj} z_j z_k^*
\ee

Again, when $\rho$ contains $1$ in every entry, the commutative sphere $C(\S^{2n - 1})$ is recovered, with complex coordinates $z_1, \ldots, z_n$ from the embedding $\S^{2n - 1} \hookrightarrow \C^n$. Further, the noncommutative sphere can be represented as a function algebra into the torus $A_\rho$.

\begin{thm}\label{thm:torusfun}(\cite{na97}, Theorem 2.5) Let $\S^{n-1}_+ = \{\vec{t} = (t_1, \ldots, t_n): 0 \leq t_i \leq 1, t_1^2 + \ldots + t_n^2 = 1\}$. Then $C(\S^{2n - 1}_\rho)$ is isomorphic to the $C^*$-algebra of continuous functions $f: \S^{n-1}_+ \to A_\rho$ which satisfy the condition that whenever $t_i = 0$, $f(\vec{t}\,) \in C^*(U_1, \ldots, U_{i - 1}, U_{i + 1}, \ldots U_n)$.
\end{thm}

This theorem is akin to writing complex coordinates in polar form $z_i = t_i u_i$ and seeing a function on the unimodular coordinates $u_i$ whenever the radial coordinates $t_i$ are fixed. Moreover, when some radial coordinate is 0, the corresponding unimodular coordinate should be irrelevant. Now, the norm on the function algebra (with operations defined pointwise) is the unique $C^*$-norm, $||f|| = \max\limits_{\vec{t} \in \S^{n-1}_+} ||f(\vec{t}\,)||_{A_\rho}$. Also, the generators $z_i$ take a simple form.
\beu
z_i(\vec{t}\,) = t_i U_i
\eeu

An enormous advantage of this formulation is that since every element of $C(\S^{2n - 1}_\rho)$ is a function on a compact space, we see the various topological joys of compact spaces (bump functions, partitions of unity, and so on) without having to pass to commutative subalgebras. Moreover, unitaries in $C(\S^{2n - 1}_\rho)$ are paths of unitaries in $A_\rho$, a well-studied object, and any element $f$ of $C(\S^{2n - 1}_\rho)$ has the property that $f(1, 0, \ldots, 0)$ belongs to the \textit{commutative} $C^*$-algebra $C^*(U_1) \cong C(\S^1)$ (and similarly for other $U_i$)! This algebra of functions also behaves well under a generalized antipodal map. In the commutative sphere $C(\S^{2n -1})$, the coordinate functions $z_i$ are odd, so in $C(\S^{2n - 1}_\rho)$ we demand that the $z_i$ be odd elements as well. This gives an isomorphism $T$ on $C(\S^{2n - 1}_\rho)$ of order $2$, or rather, a nontrivial $\Z_2$ action.
\be\label{action}
T: C(\S^{2n - 1}_\rho) \to C(\S^{2n - 1}_\rho)
\ee
\beu
z_i \mapsto -z_i
\eeu

This $\Z_2$ action on $C(\S^{2n - 1}_\rho)$ is equivalent to the pointwise action on $A_\rho$, as verified by checking on the generators, so there is no harm in also calling this map $T$. When the parameter matrix $\rho$ is relevant we will replace $T$ by $T_\rho$, and we may also apply $T_\rho$ on matrix algebras over $C(\S^{2n - 1}_\rho)$ entrywise.

\begin{prop}\label{prop:one}
Suppose $F$ is a matrix in $M_{2k - 1}(C(\S^{2n - 1}_\rho))$, $2n - 1 \geq 3$. Then $F$ cannot be both invertible and odd (i.e., odd in every entry).
\end{prop}
\begin{proof}
Suppose $F$ is invertible and odd. Then for $\vec{t} = (1, 0, \ldots, 0)$, $F(\vec{t}\,)$ is a matrix of odd dimension over $C^*(U_1) \cong C(\S^1)$ such that each entry is an odd function. Its $K_1$ class over $C(\S^1)$ is determined by $\det(F(\vec{t}\,))$, which is an odd, nowhere vanishing function $\S^1 \to \C$. By the Borsuk-Ulam theorem, this function has odd winding number, so $F(\vec{t}\,)$ is equivalent to $U_1^a$ in $K_1(C^*(U_1))$, where $a$ is odd, and certainly $F(\vec{t} \,)$ and $U_1^a$ are also equivalent in $K_1(C^*(U_1, U_2))$. Similarly, if $\vec{s} = (0, 1, 0, \ldots, 0)$, $F(\vec{s}\,)$ is equivalent to $U_2^b$ in $K_1(C^*(U_1, U_2))$, where $b$ is odd. However, there is also a path connecting $\vec{s}$ and $\vec{t}$ within $\{ \vec{r} \in \S^{n-1}_+:r_i = 0 \textrm{ for } i \geq 3\}$, so $F(\vec{s}\,)$ and $F(\vec{t}\,)$ are in the same component of invertibles over $C^*(U_1, U_2)$, which is isomorphic to a 2-dimensional quantum torus. This contradicts the fact that $U_1^a$ and $U_2^b$ are inequivalent in $K_1(C^*(U_1, U_2))$ when $a$ or $b$ is nonzero (see $\cite{pi80}$ for when the 2-torus $C^*(U_1, U_2)$ is given by an irrational rotation; the result on the rational torus follows from a homomorphism $C^*(U_1, U_2) \to M_p(C(\T^2))$ found, for example, in \cite{ma12}).
\end{proof}

\begin{prop}
There are no nontrivial projections in $C(\S^{2n - 1}_\rho)$.
\end{prop}
\begin{proof}
Natsume-Olsen spheres each admit a faithful, continuous trace $\tau$, developed in \cite{na97} by integrating the usual trace on $A_\rho$ over a Borel probability measure. We may extend $\tau$ as a linear map on $M_k(C(S^{2n - 1}))$ in the usual way by summing over the diagonal, and since this map is invariant under unitary conjugation, this allows us to view $\tau$ as a function on $K_0(C(\S^{2n - 1}_\rho))$. Now, $K_0(C(\S^{2n - 1}_\rho))$ is generated by the trivial projection $1$, so the only possible values of the trace on projections in $M_k(C(S^{2n - 1}_\rho))$ are integers. However, faithfulness implies that any nontrivial projection in $C(\S^{2n - 1}_\rho)$ must have trace in $(0, 1)$.
\end{proof}

Moreover, even though $C(\S^{2n - 1}_\rho)$ is often a noncommutative algebra, any element with a one-sided inverse always has a two-sided inverse. This is a property that distinguishes $M_n(\C)$ from $B(H)$ when $H$ is an infinite-dimensional Hilbert space, and in general a $C^*$-algebra $A$ is called \textit{finite} if whenever $x, y \in A$ have $xy = 1$, it follows that $yx = 1$. If $A$ also has the property that $M_k(A)$ is finite for all $k \in \Z^+$, then $A$ is called \textit{stably finite} (this does not follow from finiteness of $A$; see \cite{cl86}).

\begin{prop}\label{prop:side} The Natsume-Olsen spheres $C(\S^{2n - 1}_\rho)$ are all stably finite.
\end{prop}
\begin{proof}
This is immediate, as $C(\S^{2n - 1}_\rho)$ has a faithful trace.
\end{proof}

\begin{cor}\label{cor:nooddinv}
If $2n - 1 \geq 3$ and $w \in C(\S^{2n - 1}_\rho)$ is odd, then $w w^*$ is not invertible.
\end{cor} 
\begin{proof}
If $w$ is odd and $w w^*$ is invertible, the previous proposition implies that $w$ is invertible. This contradicts Proposition \ref{prop:one} for $2k - 1 = 1$.
\end{proof}

In the commutative case, there is no odd, nowhere vanishing function $F: \S^3 \to \R^3$. By identifying $\R^3 \cong \C \oplus \R$, we see that if $w, x \in C(\S^3)$ are odd and $x$ is self-adjoint, $|x|^2 + |w|^2 = x^2 + w w^*$ cannot be invertible. The above corollary makes a somewhat similar claim in the noncommutative sphere when $2n - 1 = 3$, but it is missing the self-adjoint odd element $x$. Further, when we try to rewrite the claim that there is no odd, nowhere vanishing $F: \S^{2n - 1} \to \R^{2n - 1}$ into a conjecture on elements of $C(\S^{2n - 1}_\rho)$, there is an abundance of ambiguity. This comes from the fact that if $s$ and $t$ are self-adjoint, $s^2 + t^2 = (s + it)(s - it) = (s + it)(s + it)^*$ only when $s$ and $t$ commute, which is the same as insisting $s + it$ is normal. The distinction means that the identifications  $\R^{2n - 1} \cong \R \oplus \bigoplus\limits_{i = 1}^{n - 1} \C$ and  $\R^{2n - 1} \cong \bigoplus\limits_{i = 1}^{2n - 1} \R$ lead to at least two separate questions on the noncommutative sphere.

\begin{ques}\label{ques:nonesa}
If $x \in C(\S^{2n - 1}_\rho)$ is odd and self-adjoint, and $w_1, \ldots, w_{n - 1} \in C(\S^{2n - 1}_\rho)$ are odd, must $x^2 + w_1 w_1^* + \ldots + w_{n - 1} w_{n - 1}^*$ fail to be invertible?
\end{ques}

\begin{ques}\label{ques:allsa}
If $f_1, \ldots, f_{2n - 1} \in C(\S^{2n - 1}_\rho)$ are odd and self-adjoint, must $f_1^2 + \ldots + f_{2n - 1}^2$ fail to be invertible?
\end{ques}

The second of these questions was posed by Taghavi in \cite{ta12}, as a general question about no particular family of noncommutative spheres. There are also similar questions formed by replacing some, but not all, of the expressions $w_i w_i^*$ with the square sum of two self-adjoint elements. However, regardless of formulation, the answer to each question is no.

\begin{thm} If $C(\S^{2n - 1}_\rho)$ is noncommutative, then Questions \ref{ques:nonesa} and \ref{ques:allsa}, and all intermediate versions, have a negative answer. \end{thm}

\begin{proof}
Decompose the generators as $z_m = x_m + i y_m$ where $x_m$ and $y_m$ are self-adjoint, and pick two generators $z_j$ and $z_k$ which do not commute. Since each $z_m$ is normal, $z_m z_m^* = x_m^2 + y_m^2$, so consider the following sum.

\be \label{thesum}
 (x_j + x_k)^2 + (y_j + y_k)^2 + \sum_{m \not\in \{j, k\}} z_m z_m^*
\ee

\noindent The $n - 2$ elements $z_m$ present in the sum are odd, and both $x_j + x_k$ and $y_j + y_k$ are self-adjoint and odd. So, (\ref{thesum}) is of the form in Question \ref{ques:nonesa}, where $w_1$ is actually self-adjoint. After replacement of every term $z_m z_m^*$ with $x_m^2 + y_m^2$, as $z_m$ is normal, (\ref{thesum}) becomes the square sum of $2 + 2(n - 2) = 2n - 2$ odd self-adjoint elements, so it can be written in the form of Question \ref{ques:allsa} where $f_{2n - 1} = 0$. For intermediate versions of the two questions, expand some (but not all) of the terms $z_m z_m^*$. To see that (\ref{thesum}) is invertible, we first rewrite $(x_ j + x_k)^2 + (y_j + y_k)^2$.

\beu \ba
	(x_j + x_k)^2 + (y_j + y_k)^2 	&= x_j^2 + x_k^2 + x_jx_k + x_k x_j + y_j^2 + y_k^2 + y_j y_k + y_k y_j	\\
						&= x_j^2 + y_j^2 + x_k^2 + y_k^2 + (x_jx_k + x_k x_j + y_j y_k + y_k y_j) \\
						&= z_j z_j^* + z_k z_k^* + (x_jx_k + x_k x_j + y_j y_k + y_k y_j) \\
\ea \eeu

\noindent This gives a simpler form for the original sum (\ref{thesum}).
\beu \ba
(x_j + x_k)^2 + (y_j + y_k)^2 + \sum_{m \not\in \{j, k\}} z_m z_m^* 	&=  (x_jx_k + x_k x_j + y_j y_k + y_k y_j) + \sum_{m = 1}^n z_m z_m^* \\
											&= (x_jx_k + x_k x_j + y_j y_k + y_k y_j) + 1
\ea \eeu

 It suffices to show $||x_jx_k + x_k x_j + y_j y_k + y_k y_j|| < 1$. First, we rewrite the components $x_j, y_j, x_k, y_k$ in terms of $z_j$ and $z_k$ via $x_m = \cfrac{z_m + z_m^*}{2}$ and $y_m = \cfrac{z_m -  z_m^*}{2i}$. Then we rearrange terms and apply the adjoint noncommutativity relation (\ref{NCadj}).

\beu \ba
x_jx_k + x_k x_j + y_j y_k + y_k y_j 	&= \frac{1}{2}[z_j z_k^* + z_j^* z_k + z_k^* z_j + z_k z_j^*] \\
						&= \frac{1}{2}[z_j z_k^* + z_j^*z_k + \rho_{kj}z_j z_k^* + \rho_{kj}z_j^* z_k] \\
						&= \frac{1 + \rho_{kj}}{2}[z_jz_k^* + z_j^*z_k]
\ea \eeu

\noindent Next, we calculate the norm of this element by viewing it as a function from $\S^{n - 1}_+$ to the noncommutative torus $A_\rho$, as in Theorem \ref{thm:torusfun}, where $z_i(\vec{t}\,) = t_i U_i$.

\beu \ba
\left|\left|\frac{1 + \rho_{kj}}{2}[z_jz_k^* + z_j^*z_k] \right|\right| &= \cfrac{|1+ \rho_{kj}|}{2}\cdot \max_{\vec{t} \in \S^{n-1}_+}\left\{||(t_jU_j)(t_kU_k)^*  + (t_jU_j)^*(t_kU_k)||_{A_\rho} \right\}\\
					&= \cfrac{|1+ \rho_{kj}|}{2} \cdot \max_{\vec{t} \in \S^{n-1}_+}\{t_jt_k\} \cdot ||U_jU_k^*  + U_j^*U_k||_{A_\rho} \\
							&\leq \cfrac{|1+ \rho_{kj}|}{2} \cdot \cfrac{1}{2} \cdot 2 \\
							&< 1
\ea \eeu
At the last step, we have used that $\rho_{kj}$ is unimodular, but not equal to $1$, as $z_j$ and $z_k$ do not commute. Finally, the sum (\ref{thesum}) is invertible.
\end{proof}

The answers to Questions \ref{ques:nonesa} and \ref{ques:allsa} (and all questions in between) for noncommutative spheres were negative, and the proof above shows that the disconnect between commutative and noncommutative sphere is quite large. In the commutative $2n - 1$ sphere, no square sum of $2n - 1$ odd self-adjoint elements is invertible, but in a sphere where just one pair of generators fails to commute, we can form an invertible square sum using only $2n - 2$ odd self-adjoint elements. Further, when $2n - 1 = 3$, this invertible sum is of the form $s^2 + t^2$, even though $(s + it)(s + it)^*$ cannot be invertible by Corollary \ref{cor:nooddinv}. In other words, $s$ and $t$ will definitely not commute.

Another version of the Borsuk-Ulam theorem does generalize to the noncommutative case: if $f: \S^k \to \S^k$ is odd and continuous, then $f$ has odd degree. The degree of a self-map of the sphere is defined in terms of top (co)homology; since $H_k(\S^k, \Z) \cong \Z$, the induced map of $f$ on top homology can be written as $f_*: \Z \to \Z$. This homomorphism is multiplication by some integer, called the degree. This is the same number associated to the induced map $f^*$ on top cohomology $H^k(\S^k; \Q) \cong \Q$. When the sphere is of odd dimension $k = 2n - 1$, information about top cohomology is present in odd $K$-theory, $K_1(C(\S^{2n - 1})) \cong K^1(\S^{2n - 1}) \cong \Z$, and the odd Chern character (\cite{bl98}, Theorem 1.6.6) almost gives an isomorphism between odd $K$-theory and odd cohomology. More precisely, $\chi^1$ is an isomorphism between their rationalizations.

\beu
\chi^1: K^1(\S^{2n - 1}) \otimes_\Z \Q \to \bigoplus_{m \textrm{ odd}} H^m(\S^{2n - 1}; \Q) = H^{2n - 1}(\S^{2n - 1}; \Q)
\eeu

The domain group is $\Z \otimes_\Z \Q \cong \Q$, and the codomain group is also $\Q$. If $f: \S^{2n - 1} \to \S^{2n - 1}$ is continuous, the induced map on $K^1(\S^{2n - 1}) \cong \Z$ is also multiplication by some integer $a$, and the Chern character will help show this integer is the same as $b = \textrm{deg}(f)$. The Chern character is a natural transformation (see \cite{mi74}), meaning for continuous $f: \S^{2n - 1} \to \S^{2n - 1}$ we are given the following commutative diagram, which is repeated on the right with identification of each group with $\Q$.
\beu
\begin{tikzpicture}
  \matrix (m) [matrix of math nodes,row sep=3em,column sep=4em,minimum width=2em]
  {
     K^1(\S^{2n - 1}) \otimes_\Z \Q & H^{2n - 1}(\S^{2n - 1}; \Q) \\
     K^1(\S^{2n - 1}) \otimes_\Z \Q &  H^{2n - 1} (\S^{2n - 1}; \Q)\\};
  \path[-stealth]
    (m-1-1) edge node [left] {$f^* \otimes id$} (m-2-1)
            edge node [above] {$\chi^1$} (m-1-2)
    (m-2-1.east|-m-2-2) edge node [above] {$\chi^1$} (m-2-2)
    (m-1-2) edge node [right] {$f^*$} (m-2-2);
\end{tikzpicture}
\begin{tikzpicture}
  \matrix (m) [matrix of math nodes,row sep=3em,column sep=4em,minimum width=2em]
  {
    \Q & \Q \\
    \Q &  \Q\\};
  \path[-stealth]
    (m-1-1) edge node [left] {$\times a$} (m-2-1)
            edge node [above] {$\chi^1$} (m-1-2)
    (m-2-1.east|-m-2-2) edge node [above] {$\chi^1$} (m-2-2)
    (m-1-2) edge node [right] {$\times b$} (m-2-2);
\end{tikzpicture}
\eeu

On the right hand diagram, the isomorphisms $\chi^1$ on the top and bottom are the same, and we conclude that $a = b$, so the degree of a map on $\S^{2n - 1}$ is the same when defined in terms of $K^1(\S^{2n - 1}) \cong K_1(C(\S^{2n - 1}))$ instead of cohomology. This is pleasant news, as the noncommutative algebras $C(\S^{2n - 1}_\rho)$ have very accesible $K_1$ groups. (For discussion of how $K$-theory is useful in topological Borsuk-Ulam results, see \cite{li84}.) The only other ingredient in a noncommutative Borsuk-Ulam conjecture is the precise role of the $\Z_2$ action. To say that a function $\phi: \S^k \to \S^k$ is odd means that $\phi$ commutes with the antipodal map $\alpha(\vec{x}) = -\vec{x}$.
\beu
\phi \circ \alpha = \alpha \circ \phi
\eeu
If $T: C(\S^k) \to C(\S^k)$ denotes the algebraic antipodal map $g \mapsto g \circ \alpha$ and $\Phi: C(\S^k) \to C(\S^k)$ denotes the homomorphism $g \mapsto g \circ \phi$, then the above equation has an algebraic reformulation.
\beu
T \circ \Phi = \Phi \circ T
\eeu

In other words, odd maps correspond to homomorphisms which commute with $T$. There is no reason that the domain or codomain of $\phi$ must be $\S^k$, or that the domain and codomain must be the same. In general, if $X$ and $Y$ have $\Z_2$ actions denoted as $x \mapsto -x$, then $\phi: X \to Y$ is odd if and only if $\Phi: C(Y) \to C(X)$ satisfies $T_X \circ \Phi = \Phi \circ T_Y$, where $T_X$ and $T_Y$ are the algebraic antipodal maps. If $T_X \circ \Phi = \Phi \circ T_Y$ and $g \in C(Y)$ is even or odd, then $T_X(\Phi(g)) = \Phi(T_Y(g)) = \Phi(\pm g) = \pm \Phi(g)$, so $\Phi(g)$ is also even or odd, and $\Phi$ preserves homogeneity. Conversely, if $\Phi$ preserves homogeneity, the same calculation shows $T_X(\Phi(g)) = \Phi(T_Y(g))$ when $g$ is even or odd, and since the homogeneous subspaces span $C(Y)$, $T_X \circ \Phi = \Phi \circ T_Y$. 

Now that we have the seen the algebraic translations of all terminology surrounding the Borsuk-Ulam theorem, we are ready to formulate a question, which is closely related to Question 4 of \cite{ta12} for the grading given by the antipodal map.

\begin{ques}
Suppose $\Phi: C(\S^{2n - 1}_\rho) \to C(\S^{2n - 1}_\omega)$ is a unital $*$-homomorphism between two Natsume-Olsen spheres of the same dimension. If $\Phi \circ T_\rho = T_\omega \circ \Phi$, where $T_\omega$ and $T_\rho$ are as in (\ref{action}), must $\Phi$ induce a nontrivial map on $K_1$?
\end{ques}

As in the commutative case, $K_1(C(\S^{2n - 1}_\rho)) \cong \Z$. In \cite{na97}, Natsume and Olsen recursively define a $2^{n - 1} \times 2^{n - 1}$ matrix $Z(n)$ that generates $K_1(C(\S^{2n - 1}_\rho))$, which they use to show $K_1(C(\S^{2n - 1}_\rho))$ is completely described by a generalized Toeplitz operator structure. We will denote this $K_1$ generator by $Z_\rho(n)$, which is defined by the ambiguous recurrence relation

\beu
Z_\rho(1) = z_1
\eeu
\be\label{NCrecursive}
Z_\rho(k + 1) := \left[ \begin{array}{cc} Z_\rho(k) & z_{k + 1}D_1 \\ -z_{k + 1}^* D_2 & Z_\rho(k)^* \\ \end{array} \right]
\ee
where $D_1$ and $D_2$ are any diagonal matrices over $\C$ which make the resulting matrix satisfy $Z_\rho(k + 1)Z_\rho(k + 1)^* = Z_\rho(k + 1)^*Z_\rho(k + 1) = (z_1 z_1^* + \ldots z_{k + 1} z_{k + 1}^*)I$. Such a choice always exists, and no matter what choices are made, $Z_\rho(n)$ will generate $K_1(C(\S^{2n - 1}_\rho))$. 

It is not difficult to give a single, continuous choice of coefficients in the recursive step by using a different argument than in \cite{na97}, and we may even choose coefficients so that if $z_1, \ldots, z_m$ generate a $(2m - 1)$-sphere $C(\S^{2n - 1}_\rho)$ with the same noncommutativity relations as the first $m$ generators $z_1, \ldots, z_{m}$ of $C(S^{2n - 1}_\omega)$, $n > m$, then $Z_\rho(k) = Z_\omega(k)$ (as formal $*$-monomial matrices) for all $k$ between $1$ and $m$. First, we demand that any 3-sphere given by $z_2 z_1 = \rho_{12} z_1 z_2$ must have the following $K_1$ generator.

\beu
Z_\rho(2) = \left[ \begin{array}{cc} z_1 & z_2 \\ -\rho_{21} z_2^* & z_1^* \end{array}\right] = \left[ \begin{array}{cc} z_1 & z_2 \\ -\overline{\rho_{12}} z_2^* & z_1^* \end{array}\right]
\eeu

Together with the convention $Z_\rho(1) = z_1$, this makes a consistent, continuous choice for all $1$ and $3$-dimensional spheres, and that choice is compatible with extending a list of generators ($z_1$ alone) to form a larger sphere (in $z_1$ and $z_2$). Now, for induction we suppose we have achieved the same for all spheres of dimension up to $2(n -1) - 1 = 2n - 3$. If $\rho$ is an $n \times n$ parameter matrix, then let $\omega$ denote the minor from removing row and column $n - 1$, and form $\gamma$ by removing row and column $n$ from $\rho$. As formal $*$-monomial matrices, $Z_\omega(n - 2) = Z_\gamma(n - 2)$ by the inductive assumption, since $C(\S^{2n - 3}_\omega)$ and $C(\S^{2n - 3}_\gamma)$ have the same noncommutativity relations on the first $n - 2$ generators. By the inductive assumption again, there is a well-defined choice of $D_1$ and $D_2$ to form $Z_\omega(n - 1)$.

\beu Z_\omega(n - 1) = \left[ \begin{array}{cc} Z_\omega(n - 2) & z_nD_1 \\ \ -z_n^* D_2 & Z_\omega(n - 2)^* \end{array}\right] = \left[ \begin{array}{cc} Z_\gamma(n - 2) & z_nD_1 \\ \ -z_n^* D_2 & Z_\gamma(n - 2)^* \end{array}\right]
\eeu
If we define $Z_\rho(k) = Z_\gamma(k)$, $1 \leq k \leq n - 1$, and then choose

\beu \ba
Z_\rho (n) &= \left[ \begin{array}{cc} Z_\rho(n - 1) & z_n(D_1 \oplus \rho_{n-1,n} D_2^*)\\ \ -z_n^* (D_2 \oplus \rho_{n, n - 1}D_1^*) & Z_\rho(n - 1)^* \end{array}\right] \\
&= \left[ \begin{array}{cc} Z_\rho(n - 1) & z_n(D_1 \oplus \rho_{n-1,n} D_2^*)\\ \ -z_n^* (D_2 \oplus \overline{\rho_{n - 1, n}} D_1^*) & Z_\rho(n - 1)^* \end{array}\right]
\ea \eeu
we reach a continuous choice of $Z_\rho(n)$. The verification of this fact is tedious and not at all illuminating, so it is omitted. Regardless, we now specify $Z_\rho(n)$ as a single matrix function of $\rho$ and summarize the result in the following proposition.

\begin{prop}
If $\rho$ is an $n \times n$ parameter matrix and $1 \leq k \leq n$, then there is a formal $*$-monomial matrix $Z_\rho(k)$ (given recursively as above) of dimension $2^{k - 1} \times 2^{k - 1}$ whose coefficients vary continuously in $\rho$. This matrix satisfies $Z_\rho(k) Z_\rho(k)^* = Z_\rho(k)^* Z_\rho(k) = (z_1 z_1^* + \ldots z_k z_k^*)I$. If $\omega$ is another parameter matrix (perhaps of different dimension) whose upper left $k \times k$ submatrix agrees with that of $\rho$, then $Z_\omega(k) = Z_\rho(k)$ as a formal $*$-monomial matrix. Moreover, $Z_\rho(n)$ gives a generator of $K_1(C(\S^{2n - 1}_\rho)) \cong \Z$.
\end{prop}
From now on, any mention of $Z_\rho(k)$ will refer to this single, continuous choice of coefficients, as in the previous proposition.
\begin{exam}
Consider a $5$-sphere with parameter matrix $\rho$ such that $z_2 z_1 = \alpha z_1 z_2$, $z_3 z_1 = \beta z_1 z_3$, and $z_3 z_2 = \gamma z_2 z_3$. We know that the $3$-sphere with generators $z_1$ and $z_2$ will have the matrix

\beu
Z_\rho(2) = \left[\begin{array}{cc} z_1 & z_2 \\ -\overline{\alpha} z_2^* & z_1^* \end{array} \right]
\eeu
as a $K_1$ generator, so this will be a building block for $Z_\rho(3)$. Now consider the $3$-sphere whose generators follow the same relations as $z_1$ and $z_3$. Its $K_1$ generator

\beu
\left[\begin{array}{cc} z_1 & z_3 \\ -\overline{\beta} z_3^* & z_1^* \end{array} \right]
\eeu
is formed from $Z_\rho(1) = z_1$ using diagonal matrices $D_1 = [1]$ in the upper right and $D_2 = [\overline{\beta}]$ in the lower left. This allows us to form $F_1 = D_1 \oplus \gamma D_2^* = \left[ \begin{array}{cc} 1 & 0 \\  0 & \gamma \beta \end{array} \right]$ and $F_2 = D_2 \oplus \overline{\gamma}D_1^* = \left[ \begin{array}{cc} \overline{\beta} & 0 \\  0 & \overline{\gamma} \end{array} \right]$, giving

\beu \ba
Z_\rho(3) 	&= \left[ \begin{array}{cc} Z_\rho(2) & z_3 F_1 \\ -z_3^* F_2 & Z_\rho(2)^* \end{array} \right] \\
	&=\left[ \begin{array}{cccc}z_1 & z_2 & z_3 & 0 \\ -\overline{\alpha} z_2^* & z_1^* & 0 & \gamma \beta z_3 \\ - \overline{\beta} z_3^* & 0 & z_1^* & -\alpha z_2 \\ 0 & -\overline{\gamma}z_3^* & z_2^* & z_1 \end{array} \right]
\ea \eeu
as the $K_1$ generator for $C(\S^5_\rho)$. One can verify that the noncommutativity relations above and the adjoint versions ($z_2^* z_1 = \overline{\alpha} z_1 z_2^*$, etc.) give that $Z_\rho(3)$ is, in fact, unitary.

\end{exam}

The algebras $C(\S^{2n - 1}_\rho)$ are obtained as Rieffel deformations $C(\S^{2n - 1})_J$, so for any fixed $n \times n$ antisymmetric $J$, $(C(\S^{2n - 1})_{tJ})_{t \in [0, 1]}$ forms a continuous field of $C^*$-algebras (\cite{ri93}, Theorem 8.13). Existence of a continuous choice of $K_1$ generators $Z_\rho(n)$ is then consistent with the following result of A. Sangha in \cite{sa11}. Here $A$ denotes a separable $C^*$ algebra with a strongly continuous $\R^n$ action, and $J$ is an antisymmetric $n \times n$ matrix.

\begin{thm}[\cite{sa11}, Theorem 4.6]
Let $h \in [0, 1]$. The evaluation map $\pi_h: \Gamma((A_{tJ})_{t \in [0, 1]}) \to A_{hJ}$ is a $KK$-equivalence.
\end{thm}

The algebra of sections $\Gamma((A_{tJ})_{t \in [0, 1]})$ over the continuous field $(A_{tJ})_{t \in [0, 1]}$ is a $C^*$-algebra that may itself be obtained through Rieffel deformation, which is useful in the proof from \cite{sa11}. Specifically, if $\sigma: \R^n \to \textrm{Aut}(A)$ denotes the original action, then equip $B = C([0,1], A)$ with the following $\R^n$ action, denoted $\beta$.

\be\label{eq:fibr}
\beta_{\vec{x}}f \in C([0, 1], A) \textrm{ is defined by } s \mapsto \sigma_{\sqrt{s}\vec{x}}(f(s))
\ee
One may then form the Rieffel deformation $B_J$, which is equipped with a $C([0, 1])-$structure $\Phi: C([0, 1]) \to Z(M(B_J))$ inherited from $B = C([0, 1], A)$. The fiber at $s \in [0, 1]$ is by definition $B_J / \{\Phi(g) \cdot b: b \in B_J, g \in C([0, 1]), g(s) = 0\}$, which is isomorphic to $A_{sJ}$. Finally, $B_J$ is shown to be a maximal algebra of cross-sections, hence the notation $\Gamma(A_{tJ})_{t \in [0, 1]}$, with the quotient maps onto the fibers denoted by $\pi_s$. The above result then says that each $\pi_s$ induces an invertible element of $KK(\Gamma((A_{tJ})_{t \in [0, 1]}), A_{sJ})$, and consequently the $K$-theory maps $(\pi_s)_*$ are isomorphisms. This is of interest when considering $\R^n$-equivariant maps.

\begin{cor}\label{cor:natdef}
Suppose $A$ and $B$ are separable $C^*$-algebras equipped with strongly continuous $\R^n$ actions. If $J$ is an antisymmetric $n \times n$ matrix and $\phi: A \to B$ is $\R^n$-equivariant, then let $\phi_J: A_J \to B_J$ denote the corresponding homomorphism on the Rieffel deformations. Then the following $K$-theory diagram commutes for $i \in \{0, 1\}$.

\begin{center}
\begin{tikzpicture}
  \matrix (m) [matrix of math nodes,row sep=3em,column sep=4em,minimum width=2em]
  {
     K_i(A) & & K_i(A_J) \\
     K_i(B) & & K_i(B_J)\\};
  \path[-stealth]
    (m-1-1) edge node [left] {$\phi_*$} (m-2-1)
            edge node [above] {$(\pi_1)_* \circ (\pi_0)_*^{-1}$} (m-1-3)
    (m-2-1.east|-m-2-2) edge node [above] {$(\pi_1)_* \circ (\pi_0)_*^{-1}$} (m-2-3)
    (m-1-3) edge node [right] {$(\phi_J)_*$} (m-2-3);
\end{tikzpicture}
\end{center}
\end{cor}
\begin{proof}
The map $\phi$ induces a homomorphism $\Gamma(\phi)$ between the section algebras by applying $\phi_{sJ}: A_{sJ} \to B_{sJ}$ fiberwise. This is equivalent to defining a homomorphism $\Phi: C([0, 1], A) \to C([0, 1], B)$ using $\phi$ pointwise, noting that $\Phi$ is itself $\R^n$ equivariant (for actions in the sense of (\ref{eq:fibr})), and examining the deformed homomorphism $\Phi_J$. We then have the following commutative diagram of homomorphisms.

\begin{center} \begin{tikzpicture}
  \matrix (m) [matrix of math nodes,row sep=3em,column sep=4em,minimum width=2em]
  {
     A & \Gamma((A_{tJ})_{t \in [0, 1]}) & A_J \\
     B & \Gamma((B_{tJ})_{t \in [0, 1]}) & B_J\\};
  \path[-stealth]
    (m-1-1) edge node [left] {$\phi$} (m-2-1)
    (m-1-2) edge node [above] {$\pi_0$} (m-1-1)
    (m-2-2) edge node [above] {$\pi_0$} (m-2-1.east|-m-2-2)
    (m-1-2) edge node [right] {$\Gamma(\phi)$} (m-2-2)
    (m-1-2) edge node [above] {$\pi_1$} (m-1-3)
    (m-2-2) edge node [above] {$\pi_1$} (m-2-3)
    (m-1-3) edge node [right] {$\phi_J$} (m-2-3);
\end{tikzpicture} \end{center}
All that remains is to push this diagram to $K$-theory, where each $(\pi_s)_*$ is an isomorphism, and to cut out the middle.
\end{proof}

This corollary, which is by no means new, shows that the isomorphisms between the $K$-theory of a Rieffel deformed algebra and the original algebra are natural for equivariant homomorphisms. This is not obvious (to me, at least!) based solely on Rieffel's construction in \cite{ri93b}, but as we have seen, it is much easier to prove with knowledge of $KK$-equivalences.

\section{A Noncommutative Borsuk-Ulam Theorem} \label{sec:BUR}

If $\Phi: C(\S^{2n - 1}_\rho) \to C(\S^{2n - 1}_\omega)$ is a unital $*$-homomorphism which respects the $\Z_2$ structure, then in particular it sends the $K_1$ generator $Z_\rho(n)$, which has each entry a multiple of $z_i$ or $z_i^*$, to another matrix which is odd in each entry. This implies that $Z_\omega(n)^* \cdot \Phi(Z_\rho(n))$ is a $2^{n - 1} \times 2^{n - 1}$ matrix of even elements. If we assume that every invertible matrix with even entries is an even integer in $K_1$, then we can conclude that $\Phi(Z_\rho(n))$ corresponds to an odd integer in $K_1$ and is therefore nontrivial. In other words, we conclude that $\Phi_*$ is nontrivial on $K_1$. Note that the $\Z_2$ map $T$ does not change the $K_1$ class of an invertible matrix, as verified by checking on the $K_1$ generator $Z_\rho(n)$. This matrix satisfies $T(Z_\rho(n)) = - Z_\rho(n)$, which is $K_1$-equivalent to $Z_\rho(n)$ by scaling $-1$ to $1$ within the nonzero constants. In other words, $T$ is orientation preserving.

\begin{lem}\label{lem:evensallrho}
If $F$ is an invertible matrix over $C(\S^{2n - 1}_\rho)$ and each entry of $F$ is even, then the $K_1$ class of $F$ is an even multiple of the generator.
\end{lem} 
\begin{proof}
The algebra $C(\S^{2n - 1})^{\Z_2} \cong C(\R\mathbb{P}^{2n - 1})$ of even functions in the commutative sphere $C(\S^{2n - 1})$ gives rise to an inclusion map $\iota: C(\S^{2n - 1})^{\Z_2} \to C(\S^{2n - 1})$ which is dual to the projection $\pi: \S^{2n - 1} \to \R\mathbb{P}^{2n - 1}$. Applying the Chern character shows that the image of $K_1(C(\R\mathbb{P}^{2n - 1})) \cong K^1(\R\mathbb{P}^{2n - 1})$ in $K_1(C(\S^{2n - 1})) \cong K^1(\S^{2n - 1}) \cong \Z$ is $2 \mathbb{Z}$ from the corresponding result in cohomology. 

The antipodal action on the commutative sphere commutes with the $\R^n$ action of coordinatewise rotation that defines $C(\S^{2n - 1}_\rho) \cong C(\S^{2n - 1})_J$ for a suitable antisymmetric $J$. As such, the fixed point subalgebra $C(\S^{2n - 1})^{\Z_2}$ is itself $\R^n$-equivariant, and we may form its Rieffel deformation $(C(\S^{2n - 1})^{\Z_2})_J$. From the inclusion map $\iota: C(\S^{2n - 1})^{\Z_2} \to C(\S^{2n - 1})$ we reach the following commutative diagram from Corollary \ref{cor:natdef}.

\begin{center}
\begin{tikzpicture}
  \matrix (m) [matrix of math nodes,row sep=3em,column sep=4em,minimum width=2em]
  {
     K_1(C(\S^{2n - 1})^{\Z_2}) & & K_1((C(\S^{2n - 1})^{\Z_2})_J) \\
     K_1(C(\S^{2n - 1}))) & & K_1(C(\S^{2n - 1})_J)\\};
  \path[-stealth]
    (m-1-1) edge node [left] {$\iota_*$} (m-2-1)
            edge node [above] {$\cong$} (m-1-3)
    (m-2-1.east|-m-2-2) edge node [above] {$\cong$} (m-2-3)
    (m-1-3) edge node [right] {$(\iota_J)_*$} (m-2-3);
\end{tikzpicture}
\end{center}

All of the groups above are $\Z$, and $\iota_*$ has range $2 \Z$, so $(\iota_J)_*$ must also have range $2\Z$. Unpacking the definitions shows that $\iota_J$ is simply the inclusion map of $C(\S^{2n - 1}_\rho)^{\Z_2}$ into $C(\S^{2n - 1}_\rho)$, completing the proof.
\end{proof}

We then reach a Borsuk-Ulam result as a corollary.

\begin{cor}\label{cor:BU}
Suppose $\Phi: C(\S^{2n - 1}_\rho) \to C(\S^{2n - 1}_\omega)$ is a unital $*$-homomorphism between two Natsume-Olsen spheres of the same dimension. If $\Phi$ is equivariant for the antipodal maps, then $\Phi$ induces a nontrivial map on $K_1$. More precisely, $\Phi_*: \Z \to \Z$ is multiplication by an odd integer.
\end{cor}
\begin{proof}
The $K_1$ generators $Z_\rho(n)$ and $Z_\omega(n)$ are $2^{n - 1} \times 2^{n - 1}$ matrices with odd entries, so $Z_\omega(n)^* \cdot \Phi(Z_\rho(n))$ is a $2^{n - 1} \times 2^{n - 1}$ matrix with even entries, which must be an even integer in $K_1 \cong \Z$. This implies that $\Phi(Z_\rho(n))$ corresponds to an odd integer, so $\Phi_*$ is nontrivial.
\end{proof}
\noindent \textit{Remark.} This gives a positive answer to Question 4 of \cite{ta12} for Natsume-Olsen spheres and the usual $\Z_2$ graded structure.

One version of the Borsuk-Ulam theorem claims there is no odd, continuous map $\S^{k} \to \S^{k - 1}$. The Natsume-Olsen spheres are only defined in odd dimension, but if one generator is required to be self-adjoint, this reduces the dimension by one. When this generator is central, the resulting $C^*$-algebra is the $\theta$-deformed even sphere $C(\S^{2m}_\theta)$; see \cite{pe13}, and for earlier discussions see \cite{co01} (for dimension four) and \cite{co02}. In the definition below, we have changed the notation of \cite{pe13} to remain consistent with the choice of coordinate $\rho$ (parameter matrix) instead of $\theta$ (antisymmetric matrix).

\begin{defn}(\cite{pe13}, Definition 2.4)
Let $\rho$ be an $n \times n$ parameter matrix. Then $C(\S^{2n}_\rho)$ is the universal, unital $C^*$-algebra generated by $z_1, \ldots, z_n$ and $x$ satisfying these relations.

\beu
\begin{array}{cccc} z_j z_j^* = z_j^* z_j  &	& &	x = x^* \\ \\z_k z_j = \rho_{jk}z_j z_k  & &	& x z_j = z_j x \\  \end{array}
\eeu
\vspace{.2cm}
\beu
x^2 + z_1 z_1^* + z_2 z_2^* + \ldots + z_n z_n^* = 1
\eeu
\end{defn}

In the above definition, $x$ commutes with the other generators. This is partly because of Lemma 2.6 in \cite{na97}; if instead one chose that $x z_j = \omega_j z_j x$ for some $\omega_j \in \C$, then it would follow that $x z_j = \overline{\omega_j} z_j x$ as well, since $x$ is self-adjoint and $z_j$ is normal. So, to avoid triviality, one should ban relations $x z_j = \omega_j z_j x$ when $\omega_j \not\in \R$. From the point of view of generators and relations alone, we do not see an immediate reason why insisting $x$ commutes with some $z_j$, but anticommutes with other $z_j$, would be flawed. Regardless, we stick with the established spheres $C(\S^{2n}_\rho)$ for this paper, in which $x$ is a central element. These even spheres have a $\Z_2$ action formed by negating every generator, just as in the odd case.

M. Peterka remarks in \cite{pe13} that $K_1(C(\S^{2n}_\rho))$ is the trivial group, as one would expect from Rieffel deformation and consideration of the commutative case. Now, the topological sphere $\S^{k - 1}$ sits inside $\S^k$ as the equator, in such a way that the antipodal maps are compatible and $\S^{k - 1}$ lies inside a contractible subset of $\S^k$. The next two definitions give algebraic versions of this topological embedding; note that the maps are automatically $K_1$-trivial since the even spheres have trivial $K_1$ groups.

\begin{defn}
Suppose $\rho$ is an $n \times n$ parameter matrix with $\rho_{in} = \rho_{ni} = 1$ for all $i$, and let $\widetilde{\rho}$ be the minor of $\rho$ formed by removing row and column $n$. Then $\pi: C(\S^{2n - 1}_\rho) \to C(\S^{2n - 2}_{\widetilde{\rho}})$ is the unique, unital $*$-homomorphism defined by $z_i \mapsto z_i$ for $1 \leq i \leq n - 1$ and $z_n \mapsto x$.
\end{defn}
\noindent \textit{Remark.} We insist $z_n$ commutes with each $z_i$ in the odd sphere because $x$ commutes with each $z_i$ in the even sphere.

\begin{defn}
Suppose $\rho$ is an $n \times n$ parameter matrix.  Then $\pi: C(\S^{2n}_\rho) \to C(\S^{2n - 1}_\rho)$ is the unique, unital $*$-homomorphism defined by $z_i \mapsto z_i$ and $x \mapsto 0$.
\end{defn}

In both cases, $\pi$ exists due to the relations defining the algebras, $\pi$ respects the $\Z_2$ structure, and $\pi$ is automatically $K_1$-trivial. This leads to the following consequence of Corollary \ref{cor:BU}.

\begin{cor}\label{cor:dimensionNCBU}
There is no unital $*$-homomorphism $\Psi: C(\S^{k - 1}_\rho) \to C(\S^{k}_\omega)$ which is equivariant for the antipodal maps.
\end{cor}
\begin{proof}
If $k = 2n$, then consider $\pi: C(\S^{2n}_\omega) \to C(\S^{2n - 1}_\omega)$. Since $\pi$ respects the antipodal maps and is $K_1$-trivial, $\Phi = \pi \circ \Psi: C(\S^{2n - 1}_\rho) \to C(\S^{2n - 1}_\omega)$ also respects the antipodal maps and is $K_1$-trivial. This contradicts Corollary \ref{cor:BU}.

If $k = 2n - 1$, then $\rho$ has dimensions $(n - 1) \times (n - 1)$. Let $\mathcal{P}$ be the $n \times n$ parameter matrix with $\rho$ in the upper left and all other entries equal to $1$, and form $\pi: C(\S^{2n-1}_{\mathcal{P}}) \to C(\S^{2n - 2}_\rho)$. Then $\Phi = \Psi \circ \pi: C(\S^{2n - 1}_{\mathcal{P}}) \to C(\S^{2n - 1}_\omega)$ contradicts Corollary \ref{cor:BU}.
\end{proof}

This corollary generalizes the Borsuk-Ulam theorem in all dimensions, as desired, and it is analogous to the $q$-deformed case of \cite{ya13}. Now, just as the map $T$ on the noncommutative sphere $C(\S^{2n - 1}_\rho)$ supplies a $\Z_2$ action which generalizes the antipodal map $(z_1, \ldots, z_n) \mapsto (-z_1, \ldots, -z_n)$ on $\S^{2n - 1}$, there is nothing stopping us from defining a similar map for higher order rotations on each coordinate $z_i$. Let $\alpha_1, \ldots, \alpha_n$ be primitive $k$th roots of unity, $k \geq 2$. Then there is a unital $*$-homomorphism 

\be\label{rotationmap}
R: C(\S^{2n - 1}_\rho) \to C(\S^{2n - 1}_\rho)
\ee
\beu
z_i \mapsto \alpha_i z_i
\eeu

\noindent which generalizes coordinatewise rotation (with the same finite order on each coordinate) on the sphere $\S^{2n - 1}$. Again, $R$ exists due to the fact that the elements $\alpha_i z_i$ satisfy relations corresponding to (\ref{normal}), (\ref{NCgen}), and (\ref{SS}), and this action is again just a remnant of the full $\R^d$ rotation which deforms $C(\S^{2n - 1})$. The $K_1$ generator $Z_\rho(n)$ is usually not homogeneous for $R$; the entries are all homogeneous, but the homogeneity class changes by entry. This differs from the $\Z_2$ case, in which we could simply observe that each entry was odd. However, a quick inductive argument on the recursive definition $Z_\rho(1) = z_1$, $Z_\rho(k + 1) = \left[\begin{array}{cc} Z_\rho(k) & z_{k + 1} D_1 \\ -z_{k + 1}^* D_2 & Z_\rho(k)^*  \end{array}\right]$ shows that independent of $\rho$, there are diagonal unitary matrices $A$ and $B$ over $\C$ with $A^k = B^k = I$ for which

\be\label{inconhom}
R(Z_\rho(n)) = A Z_\rho(n) B
\ee
holds. These matrices encode the homogeneity classes of the different entries of $Z_\rho(n)$.  Further, since $A$ and $B$ have scalar entries, $R(Z_\rho(n))$ and $Z_\rho(n)$ are equivalent in $K_1(C(\S^{2n - 1}_\rho))$. It follows that $R$ preserves the $K_1$ class of any invertible matrix. As usual, we have extended $R$ to matrix algebras by entrywise application.

Because $Z_\rho(n)$ does not have consistent homogeneity among its entries, it is not immediately clear how to start with a unital $*$-homomorphism $\Phi: C(\S^{2n - 1}_\rho) \to C(\S^{2n - 1}_\omega)$ which respects a rotation and form an element which is fixed by that rotation. This was the main trick of the previous results: if $\Phi$ respects the antipodal map, then since $Z_\rho(n)$ and $Z_\omega(n)$ are odd, $Z_\omega(n)^* \cdot \Phi(Z_\rho(n))$ is even. Equation (\ref{inconhom}) includes a matrix multiplication on both sides, so $R(Z_\omega(n)^* \cdot \Phi(Z_\rho(n))) = B^* Z_\omega(n)^* \Phi(Z_\rho(n)) B$, and there is still a scalar matrix conjugation present. Because of this complication, we should examine the fixed point subalgebras of the various actions

\beu
R_U: M \mapsto U^* R(M) U
\eeu
for unitaries $U \in \mathcal{U}_d(\C)$ whose orders divide $k$ (the order of $R$). When $U$ is fixed, but we wish to increase the dimension of $M$, we allow $M \in M_{qd}(C(\S^{2n - 1}_\rho))$ and let $R_U$ act on each $d \times d$ block, or equivalently apply a conjugation by a diagonal block matrix of $q$ copies of $U$. 

It is immediate that there is a $qd \times qd$ invertible matrix, fixed by $R_U$, whose $K_1$ class is represented by $k$ in the commutative sphere $C(\S^{2n - 1})$. Indeed, take a $K_1$ generator $G$ of size $qd \times qd$, scale $G$ by a scalar unitary to assign the identity matrix at a pole, and then form a continuous path of invertibles that starts with $G$ and ends with a matrix $H$ that assigns the identity outside of a small neighborhood of the opposite pole. If the neighborhood is small enough that it does not intersect any of its images under the $\Z_k$ rotation $R$, then the product $G \cdot R_U(G) \cdots {R_U}^{k - 1}(G)$ will commute and produce an $R_U$-invariant matrix with $K_1$ class equal to $k \in \Z$. This matches with our intuition from the antipodal map, where even invertibles were assigned even integers.

\begin{lem}\label{lem:commZnfixed}
If $U \in \mathcal{U}_d(\C)$ is a scalar unitary whose order divides $k$ (the order of the rotation $R$), and $M \in GL_d(C(\S^{2n - 1}))$ is an invertible matrix over the commutative sphere with $R_U(M) = M$, then the $K_1(C(\S^{2n - 1}))$ class of $M$ is in $k \Z$. Further, the fixed point subalgebra $M_d(C(\S^{2n - 1}))^{R_U}$ has $K_1$ group isomorphic to $\Z$.
\end{lem}
\begin{proof}

Let $X_n = \{(z_1, \ldots, z_n) \in \S^{2n - 1}: z_n = 0 \textrm{ or } \textrm{Arg}(z_n) \textrm{ is a }k\textrm{th root of unity} \}$, so $X_1$ is a finite set and $X_n$, $n \geq 2$, is a union of $k$ closed balls $\overline{\mathbb{B}^{2n - 2}}$ which intersect only on their boundaries. In any case, $X_n$ is an invariant set for any $k$th order coordinate rotation. Let $J_n$ denote the ideal of the fixed point subalgebra $M_d(C(\S^{2n - 1}))^{R_U}$ that consists of matrix functions vanishing on $X_n$. Now, $J_n$ is isomorphic to $C_0(\mathbb{B}^{2n -1})$, since $\S^{2n - 1} \setminus X_n$ is $k$ disjoint copies of $\mathbb{B}^{2n - 1}$ which are orbits of a single ball under $R$. This produces an exact sequence from part of the six-term sequence for $J_n$.

\be\label{eq:theone}
K_1(C_0(\mathbb{B}^{2n - 1})) \xrightarrow{\psi} K_1(M_d(C(\S^{2n - 1}))^{R_U}) \to K_1(M_d(C(X_n))^{R_U})
\ee

We induct on the claim that the final group $K_1(M_d(C(X_n))^{R_U})$ in the sequence is trivial. For the base case $n = 1$, this is trivial for all choices of $R$ and $U$ (of appropriate order) because $X_1$ has $k$ points, and invariant functions on $X_1$ are determined by values at only one point. Now, whenever we know the final group of (\ref{eq:theone}) is trivial, this implies the first map $\psi$ is surjective, and $K_1(M_d(C(\S^{2n - 1}))^{R_U})$ is the surjective image of a cyclic group, making it cyclic as well. Any image of $\psi$ may always be written in the form of a commuting product $G \cdot R_U(G) \cdots {R_U}^{k - 1}(G)$, where $G$ assigns the identity matrix on all but one component of $\S^{2n - 1} \setminus X_n$. All elements of the product are $K_1(C(\S^{2n - 1}))$ equivalent, meaning the product's class in $K_1(C(\S^{2n - 1}))$ must lie in $k \Z$. There is always an example of a $qd \times qd$ matrix $M$ which is $R_U$-invariant, invertible, and represented by $k \not= 0$ in $K_1(C(\S^{2n - 1}))$, so this implies that the induced map $K_1(M_d(C(\S^{2n - 1}))^{R_U}) \to K_1(C(\S^{2n - 1}))$ from inclusion is an injective map between \textit{infinite} cyclic groups, with image exactly $k \Z$.

To complete the induction, assume for a fixed $n$ that the final group of $(\ref{eq:theone})$ is trivial for all coordinate rotations and unitaries $U$ of the appropriate order. Let $S$ be a rotation on $C(\S^{2n + 1})$ of order $k \geq 2$, with $R$ denoting the rotation on $C(\S^{2n - 1})$ from restricting $S$ via the inclusion $(z_1, \ldots, z_n) \mapsto (z_1, \ldots, z_n, 0)$. Note that since $X_{n+1}$ is the union of $k$ copies of $\overline{\mathbb{B}^{2n}}$ which overlap only on their boundaries, it is not hard to show that $M_d(C(X_{n+1}))^{S_U}$ is isomorphic to $\{F \in M_d(C(\overline{\mathbb{B}^{2n}})): F|_{\S^{2n -1}} \textrm{ is invariant under } R_U\}$. Again, examine part of a six term sequence.

\beu
K_1(C_0(\mathbb{B}^{2n})) \to K_1(M_d(C(X_{n + 1}))^{S_U}) \xrightarrow{\phi} K_1(M_d(C(\S^{2n - 1}))^{R_U})
\eeu

The inductive assumption shows that the final group is infinite cyclic and realized as an injective image into $K_1(C(\S^{2n - 1}))$ via the obvious map. This immediately implies that $\phi$ is trivial, since every image of $\phi$ comes from the boundary data of a function on $\overline{\mathbb{B}^{2n}}$. Since $K_1(C_0(\mathbb{B}^{2n}))$ is also trivial, it follows that $K_1(M_d(C(X_{n + 1}))^{S_U})$ is trivial, and the induction is complete.
\end{proof}

The above lemma is what one would expect given the $\Z_2$ case, where the additional conjugation by $U$ is treated as merely a technical annoyance. If $U$ is the identity matrix, the computations include terms for the odd $K$-theory of lens spaces. Moreover, the role of functions on a closed ball with boundary symmetry is somewhat reminiscent of the discussion in section 6.2 of \cite{ma03} (which proves a generalization from \cite{do83} of the Borsuk-Ulam theorem to other free actions by groups on $\S^k$), although the context and conclusions are different.

\begin{thm}\label{thm:ZnNCcase}
If $U \in \mathcal{U}_d(\C)$ is a scalar unitary whose order divides $k$ (the order of the rotation $R$), and $M \in GL_d(C(\S^{2n - 1}_\rho))$ is an invertible matrix with $R_U(M) = M$, then the $K_1(C(\S^{2n - 1}_\rho))$ class of $M$ is in $k \Z$. 
\end{thm}
\begin{proof}
The action $R_U$ commutes with the (entrywise) $\R^d$ action on $M_d(C(\S^{2n - 1}))$, so we may deform the fixed point subalgebra $M_d(C(\S^{2n - 1}))^{R_U}$ using the restricted action. With the previous lemma establishing the commutative case, the proof is essentially identical to that of Lemma \ref{lem:evensallrho}.
\end{proof}

The unitary conjugation present in this section's results serves to solve the following dilemma. For rotations of order $k > 2$, the $K_1$ generator of the noncommutative sphere is not homogeneous, so when we consider a homomorphism $\Phi$ between two spheres, it is difficult to construct a matrix fixed by $R$. We can, however, easily find a matrix fixed by a specific action $R_U(M) = U R(M) U^*$, as in the following corollary.

\begin{cor}\label{cor:ZnBU}
Suppose $\Phi: C(\S^{2n - 1}_\rho) \to C(\S^{2n - 1}_\omega)$ is a unital $*$-homomorphism satisfying $R \circ \Phi = \Phi \circ R$. Here $R$ (on either sphere) denotes a rotation map defined in (\ref{rotationmap}) for the same list of primitive $k$th roots of unity $\alpha_1, \ldots, \alpha_n$, where $k \geq 2$. Then $\Phi_*$ is nontrivial on $K_1$. Specifically, it is given by multiplication by an integer congruent to $1$ mod $k$.
\end{cor}
\begin{proof}
Equation (\ref{inconhom}) gives that $R(Z_\rho(n)) = A Z_\rho(n) B$ and $R(Z_\omega(n)) = A Z_\omega(n) B$, where $A$ and $B$ are diagonal unitaries with scalar entries that do not change with the sphere parameter, and further $A^k = B^k = I$. Since $\Phi$ is a unital $*$-homomorphism and respects the rotation maps, this implies that

\beu \ba
R(Z_\omega(n)^* \cdot \Phi(Z_\rho(n))) 	&= R(Z_\omega(n))^* \cdot \Phi ( R( Z_\rho(n))) \\
							&= (A Z_\omega(n) B)^* \cdot \Phi( A Z_\rho(n) B) \\
							&= B^* Z_\omega(n)^* A^* \cdot A \Phi(Z_\rho(n)) B \\
							&= B^* (Z_\omega(n)^* \cdot \Phi(Z_\rho(n))) B \\
\ea \eeu
and $Z_\omega(n)^* \cdot \Phi(Z_\rho(n))$ is fixed by the operation $M \mapsto B R(M) B^*$. Since $B$ is a unitary over $\C$ with $B^k = I$, by the previous theorem the $K_1$ class of $Z_\omega(n)^* \cdot \Phi(Z_\rho(n))$ is in $k\Z$. Finally, the $K_1$ class of $\Phi(Z_\rho(n))$ is congruent to $1$ mod $k$.
\end{proof}

The above corollary is a noncommutative version of the following fact: if $\alpha_1, \ldots, \alpha_n$ are primitive $k$th roots of unity ($k \geq 2$), and $f: \S^{2n - 1} \to \S^{2n - 1}$ is continuous and respects the rotation $(z_1, \ldots, z_n) \mapsto (\alpha_1 z_1, \ldots, \alpha_n z_n)$, then $f$ is homotopically nontrivial. Just as in the $\Z_2$ case, there is a consequence regarding spheres of different dimensions. Specifically, there exists no $g: \S^{2n - 1} \to \S^{2n - 3}$ which is continuous and equivariant for the rotations $(z_1, \ldots, z_n) \mapsto (\alpha_1 z_1, \ldots, \alpha_n z_n)$ and $(z_1, \ldots, z_{n-1}) \mapsto (\alpha_1 z_1, \ldots, \alpha_{n-1} z_{n-1})$. This result is only necessary to state when $k$ is an odd prime (which gives the result when $k$ is not prime, but has an odd prime divisor), as we have already stated the usual Borsuk-Ulam theorem for $k = 2$. In \cite{ta90}, Z. Tang showcases a proof of this topological result using the reduced $K$-theory of lens spaces; in contrast, our proofs work entirely with $K_1$. The nonexistence of equivariant $g: \S^{2n - 1} \to \S^{2n - 3}$ can also be shown using homology; see \cite{ko86} for this type of approach (and a generalization). For noncommutative spheres, the associated result is as follows.

\begin{cor}\label{cor:orderBUdim}
Suppose $k \geq 2$ and $\alpha_1, \ldots, \alpha_n$ are primitive $k$th roots of unity. If $\Psi: C(\S^{2n - 3}_\omega) \to C(\S^{2n - 1}_\rho)$ is a unital $*$-homomorphism, then $R \circ \Psi \not= \Psi \circ R^\prime$, where $R$ denotes the rotation map for $\alpha_1, \ldots, \alpha_n$ and $R^\prime$ denotes the rotation map for the first $n - 1$ of these scalars $\alpha_1, \ldots, \alpha_{n - 1}$.
\end{cor}
\begin{proof} 
Suppose $\Psi$ satisfies $R \circ \Psi = \Psi \circ R^\prime$. Choose an $n \times n$ parameter matrix $\Omega$ which contains $\omega$ in the upper left, and let $\pi: C(\S^{2n - 1}_\Omega) \to C(\S^{2n - 3}_\omega)$ be the map defined by $z_i \mapsto z_i$ for $i \leq n - 1$ and $z_n \mapsto 0$. The map $\pi$ is $K_1$-trivial because $\pi(Z_\Omega(n)) = \left[ \begin{array}{cc} Z_\omega(n - 1) & 0 \\ 0 & Z_\omega(n - 1)^* \end{array}\right]$, which is equivalent in $K_1(C(\S^{2n - 3}_\omega))$ to $Z_\omega(n - 1) Z_\omega(n - 1)^* = I$, the trivial element. Moreover, the homogeneity classes of $\pi(z_i)$ show that $R^\prime \circ \pi = \pi \circ R$, so $\Phi = \Psi \circ \pi: C(\S^{2n - 1}_\Omega) \to C(\S^{2n - 1}_\rho)$ is $K_1$-trivial and has $\Phi \circ R = R \circ \Phi$. This contradicts Corollary \ref{cor:ZnBU}.
\end{proof}

If $k = 2$, this is not the full strength of the Borsuk-Ulam theorem, which is instead found in Corollary \ref{cor:dimensionNCBU}. Similarly, if $k$ is even, equivariant maps for order $k$ rotations are also equivariant for the antipodal map, so the antipodal results are often preferable. However, if $k$ is odd, the map $R$ on $C(\S^{2n - 1}_\rho)$ relies heavily on the complex coordinates $z_i$, so the results cannot be stated using the even dimensional spheres. In this sense both Corollary \ref{cor:ZnBU} and Corollary \ref{cor:orderBUdim} may be viewed as full-strength noncommutative $\Z_k$ Borsuk-Ulam theorems when $k$ is odd.

\section{Acknowledgments}

It is an understatement to say none of this would have been possible without my advisors, John McCarthy and Xiang Tang. I must also thank Xiang in particular for explaining \cite{ta90}, as I am unable to read Chinese. Further, my time spent browsing questions and answers on Math Stack Exchange and Math Overflow has proved invaluable; I am deeply indebted to the userbases of both sites. Without these factors, my foray into algebraic topology and noncommutative geometry would have been far less smooth. Moreover, I would also like to thank the reviewer for significant simplification of arguments.

\bibliography{references}

\begin{thebibliography}{10}

\bibitem{ba15}
{P}aul~{F}. {B}aum, {L}udwik {D}abrowski, and {P}iotr~{M.} {H}ajac.
\newblock {N}oncommutative {B}orsuk-{U}lam type {C}onjectures.
  arxiv:1502.05756.

\bibitem{bl98}
{B}ruce {B}lackadar.
\newblock {\em {$K$}-theory for operator algebras}, volume~5 of {\em
  {M}athematical {S}ciences {R}esearch {I}nstitute {P}ublications}.
\newblock {C}ambridge {U}niversity {P}ress, {C}ambridge, {S}econd edition,
  1998.

\bibitem{cl86}
{N}.~{P}. {C}larke.
\newblock A finite but not stably finite {$C^\ast$}-algebra.
\newblock {\em {P}roc. {A}mer. {M}ath. {S}oc.}, 96(1):85--88, 1986.

\bibitem{co02}
{A}lain {C}onnes and {M}ichel {D}ubois {V}iolette.
\newblock {N}oncommutative finite-dimensional manifolds. {I}. {S}pherical
  manifolds and related examples.
\newblock {\em {C}omm. {M}ath. {P}hys.}, 230(3):539--579, 2002.

\bibitem{co01}
{A}lain {C}onnes and {G}iovanni {L}andi.
\newblock {N}oncommutative manifolds, the instanton algebra and isospectral
  deformations.
\newblock {\em {C}omm. {M}ath. {P}hys.}, 221(1):141--159, 2001.

\bibitem{da15}
{L}udwik {D}abrowski.
\newblock {T}owards a noncommutative brouwer fixed-point theorem.
  arxiv:1504.03588.

\bibitem{do83}
{A}lbrecht {D}old.
\newblock {S}imple proofs of some {B}orsuk-{U}lam results.
\newblock In {\em {P}roceedings of the {N}orthwestern {H}omotopy {T}heory
  {C}onference ({E}vanston, {I}ll., 1982)}, volume~19 of {\em {C}ontemp.
  {M}ath.}, pages 65--69. {A}mer. {M}ath. {S}oc., {P}rovidence, {R}{I}, 1983.

\bibitem{do98}
{R}onald~{G}. {D}ouglas.
\newblock {\em {B}anach algebra techniques in operator theory}, volume 179 of
  {\em {G}raduate {T}exts in {M}athematics}.
\newblock {S}pringer-{V}erlag, {N}ew {Y}ork, {S}econd edition, 1998.

\bibitem{ha02}
{A}llen {H}atcher.
\newblock {\em {A}lgebraic topology}.
\newblock {C}ambridge {U}niversity {P}ress, {C}ambridge, 2002.

\bibitem{ko86}
{T}eiichi {K}obayashi.
\newblock {T}he {B}orsuk-{U}lam theorem for a {$Z_q$}-map from a {$Z_q$}-space
  to {$S^{2n+1}$}.
\newblock {\em {P}roc. {A}mer. {M}ath. {S}oc.}, 97(4):714--716, 1986.

\bibitem{li84}
{A}runas {L}iulevicius.
\newblock {B}orsuk-{U}lam theorems and {$K$}-theory degrees of maps.
\newblock In {\em {A}lgebraic topology, {A}arhus 1982 ({A}arhus, 1982)}, volume
  1051 of {\em {L}ecture {N}otes in {M}ath.}, pages 610--619. {S}pringer,
  {B}erlin, 1984.

\bibitem{ma12}
{M}atilde {M}arcolli and {C}hristopher {P}erez.
\newblock {C}odes as fractals and noncommutative spaces.
\newblock {\em {M}ath. {C}omput. {S}ci.}, 6(3):199--215, 2012.

\bibitem{ma03}
{J}i{\v{r}}{\'{\i}} {M}atou{\v{s}}ek.
\newblock {\em {U}sing the {B}orsuk-{U}lam theorem}.
\newblock {U}niversitext. {S}pringer-{V}erlag, {B}erlin, 2003.
\newblock {L}ectures on topological methods in combinatorics and geometry,
  {W}ritten in cooperation with {A}nders {B}j{\"o}rner and {G}{\"u}nter {M}.
  {Z}iegler.

\bibitem{ma91}
{K}engo {M}atsumoto.
\newblock {N}oncommutative three-dimensional spheres.
\newblock {\em {J}apan. {J}. {M}ath. ({N}.{S}.)}, 17(2):333--356, 1991.

\bibitem{mi74}
{J}ohn~{W}. {M}ilnor and {J}ames~{D}. {S}tasheff.
\newblock {\em {C}haracteristic classes}.
\newblock {P}rinceton {U}niversity {P}ress, {P}rinceton, {N}. {J}.;
  {U}niversity of {T}okyo {P}ress, {T}okyo, 1974.
\newblock {A}nnals of {M}athematics {S}tudies, {N}o. 76.

\bibitem{na97}
{T}. {N}atsume and {C}.~{L}. {O}lsen.
\newblock {T}oeplitz {O}perators on {N}oncommutative {S}pheres and an {I}ndex
  {T}heorem.
\newblock {\em {I}ndiana {U}niv. {M}ath. {J}.}, 46(4):1055--1112, 1997.

\bibitem{pe13}
{M}ira~{A}. {P}eterka.
\newblock {F}initely-{G}enerated {P}rojective {M}odules over the
  {$\theta$}-{D}eformed 4-{S}phere.
\newblock {\em {C}omm. {M}ath. {P}hys.}, 321(3):577--603, 2013.

\bibitem{ph87}
{N}.~{C}hristopher {P}hillips.
\newblock {\em {E}quivariant {$K$}-theory and freeness of group actions on
  {$C^*$}-algebras}, volume 1274 of {\em {L}ecture {N}otes in {M}athematics}.
\newblock {S}pringer-{V}erlag, {B}erlin, 1987.

\bibitem{ph09}
{N}.~{C}hristopher {P}hillips.
\newblock {F}reeness of actions of finite groups on {$C^*$}-algebras.
\newblock In {\em {O}perator structures and dynamical systems}, volume 503 of
  {\em {C}ontemp. {M}ath.}, pages 217--257. {A}mer. {M}ath. {S}oc.,
  {P}rovidence, {R}{I}, 2009.

\bibitem{pi80}
{M}. {P}imsner and {D}. {V}oiculescu.
\newblock {E}xact sequences for {$K$}-groups and {E}xt-groups of certain
  cross-product {$C^{\ast} $}-algebras.
\newblock {\em {J}. {O}perator {T}heory}, 4(1):93--118, 1980.

\bibitem{ri93}
{M}arc~{A}. {R}ieffel.
\newblock {D}eformation quantization for actions of {${\bf R}^d$}.
\newblock {\em {M}em. {A}mer. {M}ath. {S}oc.}, 106(506):x+93, 1993.

\bibitem{ri93b}
{M}arc~{A}. {R}ieffel.
\newblock {$K$}-groups of {$C^*$}-algebras deformed by actions of {${\bf
  R}^d$}.
\newblock {\em {J}. {F}unct. {A}nal.}, 116(1):199--214, 1993.

\bibitem{sa11}
{A}mandip {S}angha.
\newblock {K}{K}-fibrations arising from {R}ieffel deformations.
  arxiv:1109.5968.

\bibitem{ta12}
{A}li {T}aghavi.
\newblock {A} {B}anach algebraic approach to the {B}orsuk-{U}lam theorem.
\newblock {\em {A}bstr. {A}ppl. {A}nal.}, pages Art. ID 729745, 11, 2012.

\bibitem{ta90}
{Z}izhou {T}ang.
\newblock {A} {$K$}-theory proof of mod $q$ {B}orsuk-{U}lam.
\newblock {\em {C}hinese {S}cience {B}ulletin}, 35(23):1834, 1990.

\bibitem{ya13}
{M}akoto {Y}amashita.
\newblock {E}quivariant comparison of quantum homogeneous spaces.
\newblock {\em {C}omm. {M}ath. {P}hys.}, 317(3):593--614, 2013.

\end{thebibliography}
\end{document}